\documentclass[12pt]{article}%

\usepackage{amssymb}
\usepackage{amsfonts}
\usepackage{amsmath}
\usepackage[nohead]{geometry}
\usepackage[singlespacing]{setspace}
\usepackage[bottom]{footmisc}
\usepackage{indentfirst}
\usepackage{endnotes}
\usepackage{graphicx}%
\usepackage{rotating}
\usepackage{verbatim}
\usepackage{setspace}
\usepackage{multirow}
\usepackage{latexsym}

\usepackage{graphicx}
\usepackage{caption}
\usepackage{subcaption}

\usepackage{amsthm}
\usepackage{makeidx}
\usepackage{fancyhdr}
\usepackage{type1cm}

 \usepackage[authoryear]{natbib}

\newcommand{\bb}{\mathrm{\bf b}}
\newcommand{\bff}{\mathrm{\bf f}}
\newcommand{\bx}{\mathrm{\bf x}}
\newcommand{\by}{\mathrm{\bf y}}

\newcommand{\bA}{\mathrm{\bf A}}
\newcommand{\ba}{\mathrm{\bf a}}
\newcommand{\bu}{\mathrm{\bf u}}
\newcommand{\bv}{\mathrm{\bf v}}
\newcommand{\be}{\mathrm{\bf e}}
\newcommand{\bB}{\mathrm{\bf B}}
\newcommand{\bZ}{\mathrm{\bf Z}}

\newcommand{\bD}{\mathrm{\bf D}}
\newcommand{\bE}{\mathrm{\bf E}}
\newcommand{\bF}{\mathrm{\bf F}}
\newcommand{\bK}{\mathrm{\bf K}}
\newcommand{\bW}{\mathrm{\bf W}}
\newcommand{\bG}{\mathrm{\bf G}}
\newcommand{\bM}{\mathrm{\bf M}}
\newcommand{\bh}{\mathrm{\bf h}}
\newcommand{\bH}{\mathrm{\bf H}}
\newcommand{\bI}{\mathrm{\bf I}}

\newcommand{\bV}{\mathrm{\bf V}}

\newcommand{\br}{\mathrm{\bf r}}
\newcommand{\bR}{\mathrm{\bf R}}

\newcommand{\bU}{\mathrm{\bf U}}
\newcommand{\bX}{\mathrm{\bf X}}
\newcommand{\bY}{\mathrm{\bf Y}}

\newcommand{\bdelta}{\mbox{\boldmath $\delta$}}

\newcommand{\bxi}{\mbox{\boldmath $\xi$}}
\newcommand{\beps}{\mbox{\boldmath $\varepsilon$}}
\newcommand{\bmu}{\mbox{\boldmath $\mu$}}
\newcommand{\bLambda}{\mbox{\boldmath $\Lambda$}}

\newcommand{\bGamma}{\mbox{\boldmath $\Gamma$}}

\newcommand{\bSigma}{\mbox{\boldmath $\Sigma$}}
\newcommand{\bOmega}{\mbox{\boldmath $\Omega$}}

\newcommand{\cov}{\mathrm{cov}}

\newcommand{\tr}{\mathrm{tr}}
\newcommand{\diag}{\mathrm{diag}}

\newcommand{\bw}{\mbox{\bf w}}

\newcommand{\var}{\mathrm{var}}

\def \Var {\mbox{Var}}
\def \diag {\mbox{diag}}
\def\unif{\mbox{Unif}}
\def\FDP{\mbox{FDP}}
\def\RE{{\rm RE}}
\newcommand{\bTheta}{\mbox{\boldmath $\Theta$}}

\newcommand{\bfeta}{\mbox{\boldmath $\eta$}}
\newcommand{\beq}{\begin{equation}}
\newcommand{\eeq}{\end{equation}}
\newcommand{\beqn}{\begin{eqnarray}}
\newcommand{\eeqn}{\end{eqnarray}}
\newcommand{\beqnn}{\begin{eqnarray*}}
\newcommand{\eeqnn}{\end{eqnarray*}}

\numberwithin{equation}{section}
\theoremstyle{plain}
\newtheorem{thm}{Theorem}[section]

\newtheorem{lem}{Lemma}[section]

\newtheorem{assum}{Assumption}[section]
\theoremstyle{definition}

\makeatletter
\def\@biblabel#1{\hspace*{-\labelsep}}
\makeatother
\begin{document}

 \title{Asymptotics of Empirical Eigen-structure for Ultra-high Dimensional Spiked Covariance Model}
\author{Jianqing Fan
\thanks{Address: Department of ORFE, Sherrerd Hall, Princeton University, Princeton, NJ 08544, USA, e-mail: \textit{jqfan@princeton.edu},
\textit{weichenw@princeton.edu}. The research was partially supported by NSF grants
DMS-1206464 and DMS-1406266 and NIH grants R01-GM072611-10
and NIH R01GM100474-04.}$\; ^\dag$ and Weichen Wang$^*$
\medskip\\{\normalsize $^*$Department of Operations Research and Financial Engineering,  Princeton University}
\medskip\\{\normalsize $^\dag$ Bendheim Center for Finance, Princeton University}}
 
\date{}

\maketitle

\sloppy

\onehalfspacing
 
\begin{abstract}
We derive the asymptotic distributions of the spiked eigenvalues and eigenvectors under a generalized and unified asymptotic regime, which takes into account the spike magnitude of leading eigenvalues, sample size, and dimensionality. This new regime allows high dimensionality and diverging eigenvalue spikes and provides new insights into the roles the leading eigenvalues, sample size, and dimensionality play in principal component analysis.  The results are proven by a technical device, which swaps the role of rows and columns and converts the high-dimensional problems into  low-dimensional ones. Our results are a natural extension of those in \cite{Pau07} to more general setting with new insights and solve the rates of convergence problems in \cite{SheSheZhuMar13}. They also reveal the biases of the estimation of leading eigenvalues and eigenvectors by using principal component analysis, and lead to a new covariance estimator for the approximate factor model, called shrinkage principal orthogonal complement thresholding (S-POET), that corrects the biases.  Our results are successfully applied to outstanding problems in estimation of risks of large portfolios and false discovery proportions
for dependent test statistics and are illustrated by simulation studies.
\end{abstract}


\pagebreak%
\doublespacing

\onehalfspacing

\section{Introduction}

Principal Component Analysis (PCA) has widely been used as a powerful tool for dimensionality reduction and data visualization.
Its theoretical properties such as the consistency and asymptotic distributions of empirical eigenvalues and eigenvectors are  challenging especially in high dimensional regime.
For the past half century substantial amount of efforts have been devoted to understanding empirical eigen-structures.  An early effort is \cite{And63} who established the asymptotic normality of
eigenvalues and eigenvectors under the classical regime with  large sample size $n$ and fixed dimension $p$. However, as dimensionality diverges at the same rate as the sample size, sample covariance matrix is a notoriously bad estimator with substantial different eigen-structure from the population one.
A lot of recent literatures make the endeavor to understand the behaviors of eigenvalues and eigenvectors under high dimensional regime where both $n$ and $p$ go to infinity. See for example \cite{BaiAroPec05, Bai99, Pau07, JohLu09, Ona12, SheSheZhuMar13} and many related papers.  For additional developments and references, see \cite{BaiSil09}.

Most of studies focus on the situations where signals are weak or semi-weak \citep{Ona12} with leading asymptotic eigenvalues bounded \citep{Pau07, BaiSil09} or slowly growing \citep{Ona12}.  However, \cite{FanLiaMin13} shows that for factor models with pervasive factors, the leading eigenvalues can grow linearly with the dimensionality and hence their corresponding eigenvectors can be consistently estimated as long as sample size diverges.  This leads to the question of how the asymptotics of engen-structure depends on the interplay of spike magnitude of leading eigenvalues, dimensionality, and sample size.  An interesting study on this topic is \cite{SheSheZhuMar13}, which focuses only on the consistency of the problem.  The question then arises naturally on the rates of convergence and asymptotic structures of empirical eigenvalues and eigenvectors.  This is the subject of this study.

In this paper, we consider a high dimensional spiked covariance model with the first several eigenvalues significantly larger than the rest. Typically, the spike part is of importance and of interest. We provides new understanding on how the spiked empirical eigenvalues and eigenvectors fluctuate around their theoretical counterparts and what their asymptotic biases are. For the spiked covariance model, three quantities play an essential role in determining the asymptotic behaviors of empirical eigen-structure: the sample size $n$, the dimension $p$, and the magnitude of leading eigenvalues $\{\lambda_j\}_{j=1}^m$. Theoretical properties of PCA have been investigated from three different perspectives.

The first angle is through a low-rank plus sparse decomposition, where the covariance matrix is perceived as the sum of a low-rank and a sparse matrix. The low-rank part contributes to the signal to be recovered whereas the sparse part serves as noise. For example in \cite{FanFanLv08}, the low-rank matrix corresponds to the dependence induced by the common factors or covariates whereas the sparse matrix corresponds to the idiosyncratic noise. In noiseless setting, \cite{CanLiMaWri11} considered the principal component pursuit and showed that it can recover the decomposition structure under the incoherence condition. \cite{ChaSanParWil11} also studied the sufficient condition for exact recovery of the low-rank and sparse matrices. The noisy decomposition recover was considered more thoroughly by \cite{AgaNegWai12}. In addition, a large amount of literature has contributed to the topic of sparse PCA, for example \cite{AmiWai09, VuLei12, BirJohNadPau13, BerRig13, Ma13}, which leverages the extra assumption on the sparsity of eigenvectors. Specifically, \cite{CaiMaWu13} studied the minimax optimal rates for estimating eigenvalues and eigenvectors of spiked covariance matrices with jointly $k$-sparse eigenvectors. This type of work assumes bounded eigenvalues, which limit the signals we can get from the data. Correspondingly, those works require additional eigenvector structure to reduce the possibility of noise accumulation such as incoherence or jointly $k$-sparse or other similar conditions. In this paper, thanks to the diverging eigenvalue regime we will consider, our conclusions will not rely on additional structure of eigenvectors, which can be hard to verify in practice.

A different line of efforts is to analyze PCA through random matrix theories, where it is
typically assumed $p/n \to \gamma \in (0, \infty)$ with bounded spike sizes. It is well known that if the true covariance matrix is identity, the empirical spectral distribution converges almost surely to the Marcenko-Pastur distribution \citep{Bai99}  and when $\gamma < 1$ the largest and smallest eigenvalues converge almost surely to $(1+\sqrt{\gamma})^2$ and $(1-\sqrt{\gamma})^2$ respectively \citep{BaiYin93, Joh01}. If the true covariance structure takes the form of a spiked matrix, \cite{BaiAroPec05} showed that the asymptotic distribution of the empirical eigenvalues exhibit an $n^{2/3}$ scaling when the eigenvalue lies below a threshold $1+\sqrt{\gamma}$, and an $n^{1/2}$ scaling when it is above the threshold. For the case where we have the regular scaling, \cite{Pau07} investigated the asymptotic behavior of the corresponding empirical eigenvectors and showed that the major part of an eigenvector which corresponds to the spiked eigenvalues is normally distributed with regular scaling $n^{1/2}$. The convergence of principal component scores under this regime was considered by \cite{LeeZouWri10}.
The same random matrix regime has also been considered by \cite{Ona12} in studying the principal component estimator for high-dimensional factor models. More recently, \cite{KolLou14a, KolLou14b} revealed a profound link of concentration bounds of empirical eigen-struecture with the effective rank defined as $\bar r = \tr(\bSigma)/\lambda_1$ \citep{Ver10}. Their results extend the regime of bounded eigenvalues to more general setting, although the asymptotic results in most cases still rely on the assumption $\bar r = o(n)$. In this paper, we consider the regime $p/(n\lambda_1) < \infty$, which implies $\bar r = O(n)$. More discussions will be given in Section \ref{sec3}.

Deviating from the classical random matrix and sparse PCA literature, we consider the ultra-high dimensional regime allowing $p/n \to \infty$. If $p/n\to \infty$, to ensure sufficiently strong signal for PCA, it is natural to also have the spike sizes go to infinity, namely, $\lambda_j \to \infty$ for the first $m$ leading eigenvalues.
This leads to the third perspective for understanding PCA from this ultra high dimensional setting. \cite{SheSheZhuMar13} adopted this point of view and considered the regime of $p/(n\lambda_j) \to c_j$ where $0 \le c_j< \infty$ for leading eigenvalues. This is more general than the bounded eigenvalue condition. Specifically if eigenvalues are bounded, we require the ratio $p/n$ converges to a bounded constant. On the other hand, if the dimension is much larger than the sample size, we offset the dimensionality by assuming increased signals. In particular, the pervasive factor model considered in economics and finance factor model corresponds to $c_j = 0$ with the pervasive leading eigenvalues $\lambda_j \asymp p$, see for example \cite{FanLiaMin13, FanLiaWan14, StoWat02, Bai03, BaiNg02}.  The weak factor model considered by \cite{Ona12} also implies $c_j = 0$, with $p/n$ bounded and $\lambda_j \asymp p^\theta$ for some $\theta \in (0, 1)$. \cite{HalMarNee05, JunMar09} started the research of high dimension low sample size (HDLSS) regime. With $n$ fixed, \cite{JunMar09} concluded that consistency of leading eigenvalues and eigenvectors is granted if $\lambda_j \asymp p^\theta$ for $\theta > 1$, which also corresponds to $c_j = 0$.
\cite{SheSheZhuMar13} revealed an interesting fact that when $c_j \ne 0$, spiked sample eigenvalues almost surely converges to a biased quantity of the true eigenvalues; furthermore the corresponding sample eigenvectors show an asymptotic conical structure. We will consider the same regime as theirs, but focus more on the asymptotic distributions of the eigen-structure, which was not covered in their paper, and under more relaxed conditions. Our results can be seen as a natural extension of \cite{Pau07} to ultra high dimensional setting.

In addition to the different regimes we take on, we also introduce a simple technique to for our technical proofs. The idea is to flip the roles of rows and columns and treat $p$ as the sample size and $n$ as the dimension. When $p$ is higher than $n$, sample covariance is clearly degenerate. Switching the roles of $n$ and $p$ allows us to utilize the existing results on eigen-structures. To be specific, if we have $n$ samples generated from $N({\bf 0}, \bD)$ where $\bD = \diag(d_1,\dots,d_p)$ is diagonal, then all the information we have is just an $n$ by $p$ data matrix with independent entries. We can simply treat the data as $p$ independent vectors of dimension $n$ each with distribution $N({\bf 0}, d_i \bI_n)$. Even when the data are not normally distributed and hence $p$ $n$-dimensional vectors are then not independent, the idea is still  powerful and leads to better understanding of relationship between high and low dimensionality. The simple trick has been used to derive asymptotic results of empirical eigenvalues in recent papers such as \cite{SheSheZhuMar13, YatAos12, YatAos13}. One of our contributions lies in successful application of the trick to study the empirical leading eigenvectors.

The rest of the paper is organized as follows. Section \ref{sec2} introduces the notations, assumptions, and an important fact which serves as basis of our proofs. The fact will help unravel the relationship between high and low dimensions. Sections \ref{sec3.1} and \ref{sec3.2} devote to the theoretical results of the sample eigenvalues and eigenvectors of the spiked covariance matrix under our asymptotic regime. In Section \ref{sec4}, we discuss several applications of the theories in the previous section. Firstly a new covariance estimator for the approximate factor model, named shrinkage principal orthogonal complement thresholding (S-POET),  is proposed which corrects the biases of empirical eigenvalues. Secondly, S-POET will be successfully applied to outstanding problems in estimation of risks of large portfolios and false discovery proportions for dependent test statistics. For both problems, the typical assumption on the signal strength of leading eigenvalues in order to deploy factor analysis is relaxed due to our new results in Section \ref{sec3}. In Section \ref{sec5}, simulations are conducted to illustrate the theoretical results at the finite sample. The proofs for Section \ref{sec3} are provided in Section \ref{sec6} and those for Section \ref{sec4} are relegated to the supplementary material.

\section{Assumptions and a simple fact}
\label{sec2}
Asssume that $\{\bY_i\}_{i=1}^n$ is a sequence of i.i.d. random variables with
zero mean and covariance matrix $\bSigma_{p\times p}$.
Let $\lambda_1,  \dots, \lambda_p$ be the eigenvalues of $\bSigma$
in descending order. We consider the spiked covariance model as follows.

\begin{assum} \label{assump1}
$\lambda_1 > \lambda_2 > \dots > \lambda_m > \lambda_{m+1} \ge \dots \ge \lambda_p > 0$, where the non-spiked eigenvalues are bounded, i.e. $c_0 \le \lambda_j \le C_0, j>m$ for constants $c_0, C_0 > 0$ and the spiked eigenvalues are well separated, i.e. $\exists \delta_0 > 0$ such that $\min_{j \le m} (\lambda_j - \lambda_{j+1})/\lambda_j \ge \delta_0$.
\end{assum}

The eigenvalues are divided into the spiked ones and bounded non-spiked ones. We do not have specific order assumptions on the leading eigenvalues nor require them to diverge. Thus, our results in Section \ref{sec3} are applicable to both bounded and diverging leading eigenvalues; if diverging, they can have different diverging rates.
For simplicity,  we only consider distinguishable eigenvalues (multiplicity 1) for the largest $m$ eigenvalues and a fixed number $m$, independent of
$n$ and $p$.

The spiked covariance model is motivated by the factor model $\by = \bB \bff + \beps$ considered by \cite{FanLiaMin13} as follows.  Assume without loss of generality that $\var(\bff) = \bI_m$, the $m \times m$ identity matrix.  Then, the model implied covariance matrix $\bSigma = \bB \bB' + \bSigma_\varepsilon$, where $\bSigma_\varepsilon = \var(\beps)$.  If the factor loadings $\{\bb_i\}$ (the transpose of rows of $\bB$) are an i.i.d. sample from a population with mean zero and covariance $\bSigma_b$, then by the law of large numbers, $ p^{-1} \bB' \bB = p^{-1} \sum_{i=1}^p \bb_i \bb_i' \to \bSigma_b$.  In other words, the eigenvalues of $\bB \bB'$ are approximately
\[
  p \lambda_1(\bSigma_b) (1+o(1)), \cdots,  p \lambda_m(\bSigma_b) (1+o(1)), 0, \cdots, 0,
\]
where $\lambda_j(\bSigma_b)$ is the $j^{th}$ eigenvalue of $\bSigma_b$.
If we assume that $\|\bSigma_\varepsilon\|$ is bounded, then by Weyl's theorem, we conclude that
\begin{equation} \label{eq2.1}
  \lambda_j = p \lambda_j(\bSigma_b) (1+o(1)),  \quad \mbox{for } j = 1, \cdots, m,
\end{equation}
and the remaining is bounded.

In the spiked covariance models, three essential factors come into play: the sample size $n$, dimension $p$ and the spikeness $\lambda_j$'s. The following relationship is assumed as in \cite{SheSheZhuMar13}.

\begin{assum}
\label{assump2}
Assume $p>n$. For the spiked part $1 \le j \le m$, $c_j = p/(n \lambda_j)$ is bounded, and for the non-spiked part, $(p-m)^{-1} \sum_{j = m+1}^p \lambda_j = \bar c + o(n^{-1/2})$.
\end{assum}

We allow $p/n \to \infty$ in any manner, though $\lambda_j$ also needs also grow fast enough to ensure bounded $c_j$.  In particular, $c_j = o(1)$ is allowed as in the factor model.
We do not assume the non-spiked eigenvalues are identical, as in most spiked covariance model literature (e.g. \cite{Pau07, JohLu09}).

By spectral decomposition, $\bSigma = \bGamma \bLambda \bGamma'$, where the orthonormal matrix $\bGamma$ is constructed by the eigenvectors of $\bSigma$ and $\bLambda = \diag(\lambda _1, \dots,\lambda _p)$. Let $\bX_i = \bGamma' \bY_i$. Since the empirical eigenvalues are invariant and the empirical eigenvectors are equivariant under an orthonormal transformation, we focus the analysis on the transformed domain of $\bX_i$ and the results can be translated into the original data.  Note that
$\var(\bX_i) = \bLambda$.  Let $\bZ_i = \bLambda^{-1/2} \bX_i$ be the elementwise standardized random vector.

\begin{assum}
\label{assump3}
$\{\bZ_i\}_{i=1}^n$ are i.i.d copies of $\bZ$.  The standardized random vector $\bZ = (Z_1, \dots, Z_p)$ is sub-Gaussian with independent entries of mean zero and variance one.  The sub-Gaussian norms of all components are uniformly bounded:
$\max_{j} \| Z_{j} \|_{\psi_2} \leq C_0$, where
$\| Z_{j} \|_{\psi_2} = \sup_{q \geq 1} q^{-1/2} (E|Z_{j}|^q)^{1/q}$.
\end{assum}

\noindent

Since $\Var(\bX_i) = \diag(\lambda _1,\lambda _2, \dots,\lambda _p)$, the first $m$ population eigenvectors are simply unit vectors $\be_1,\be_2,\dots, \be_m$. Denote the $n$ by $p$ transformed data matrix by ${\bX} = (\bX_1, \bX_2, \dots, \bX_n)'$. Then the sample covariance matrix is
\[
{\hat \bSigma}_{p \times p}= \frac{1}{n} \bX' \bX = \frac{1}{n} \sum_{i=1}^n \bX_i \bX_i'\,,
\]
whose eigenvalues are denoted as $\hat \lambda_1, \hat \lambda_2, \dots, \hat \lambda_p$ ($\hat \lambda_j = 0$ for $j > n$) with corresponding eigenvectors $\hat\bxi_1,\hat \bxi_2,\dots, \hat\bxi_p$. Note that the empirical eigenvectors of data $\bY_i$'s are $\hat \bxi_j^{(Y)} = \bGamma \hat\bxi_j$.

Let $\bZ_j$ be the $j^{th}$ column of the standardized $\bX$.  Then each $\bZ_j$ has i.i.d sub-Gaussian entries with zero mean and unit variance. Exchanging the role of rows and columns, we get the $n$ by $n$ Gram matrix
\[
\widetilde \bSigma_{n \times n}  = \frac{1}{n} {\bX} {\bX}' = \frac{1}{n} \sum \limits_{j=1}^p \lambda_j \bZ_j \bZ_j'\,,
\]
with the same nonzero eigenvalues $\hat \lambda_1, \hat \lambda_2, \dots, \hat \lambda_n$ as $\hat\bSigma$ and the corresponding eigenvectors $\bu_1,\bu_2,\dots, \bu_n$.  It is well known that for $i=1,2,\dots,n$
\beq
\label{VecRel}
\hat\bxi_i  = (n \hat \lambda_i)^{-1/2} \bX' \bu_i \;\; \text{and} \;\; \bu_i =  (n \hat \lambda_i)^{-1/2} \bX \hat\bxi_i \,,
\eeq
while the other eigenvectors of $\hat \bSigma$ constitute a $(p-n)$-dimensional orthogonal complement of  $\hat\bxi_1,\dots,\hat\bxi_n$.

By using this simple fact, for the specific case with $c_0 = C_0 = 1$ in Assumption \ref{assump1}, $\lambda_j = 1$ for $j > m$ in Assumption \ref{assump2}, and Gaussian data in Assumption \ref{assump3}, \cite{SheSheZhuMar13} showed that
\[
\frac{\hat \lambda_j}{\lambda_j} \overset{\text{a.s.}} \to 1+ c_j\,, \;1\le j \le m \,;
\]
and
\[
\left|  \langle \hat\bxi_j, \be_j \rangle \right| \overset{\text{a.s.}} \to  (1+ c_j)^{-\frac 1 2} \,,
\]
where $\langle \ba,\bb \rangle$ denotes the inner product of two vectors. However, they fail to establish any results on convergence rates or asymptotic distributions of the empirical eigen-structure.  This motivates the current paper.

The aim of this paper is to establish the asymptotic normality of the empirical eigenvalues and eigenvectors under more relaxed conditions. Our results are a natural extension of  \cite{Pau07} to more general setting with new insights, where the asymptotic normality of sample eigenvectors is derived using complicated random matrix techniques for Gaussian data under the regime of $p/n \to \gamma \in [0,1)$. Compared to them, our proof, based on the relationship (\ref{VecRel}), is much simpler and insightful for understanding the behavior of ultra high dimensional PCA.

Here are some notations that we will use in the paper.
For a general matrix $\bM$, we denote its matrix entry-wise max norm as $\|\bM\|_{\max} = \max_{i,j}\{|M_{i,j}|\}$ and define the quantities $\|\bM\| = \lambda_{\max}^{1/2}(\bM'\bM)$, $\|\bM\|_F = (\sum_{i,j} M_{i,j}^2)^{1/2}$, $\|\bM\|_{\infty} = \max_{i} \sum_j |M_{i,j}|$ to be its spectral, Frobenius and induced $\ell_{\infty}$ norms.
If $\bM$ is symmetric, we define $\lambda_{j}(\bM)$ to be the $j^{th}$ largest eigenvalue of $\bM$ and $\lambda_{\max}(\bM)$, $\lambda_{\min}(\bM)$ to be the maximal and minimal eigenvalues respectively. We denote $\tr(\bM)$ as the trace of $\bM$.
For any vector $\bv$, its $\ell_2$ norm is represented by $\|\bv\|$ while $\ell_1$ norm is written as $\|\bv\|_1$. We use $\diag(\bv)$ to denote the diagonal matrix with the same diagonal entries as $\bv$. For two random vectors $\ba, \bb$ of the same length, we say $\ba = \bb + O_P(\delta)$ if $\|\ba - \bb\| = O_P(\delta)$ and $\ba = \bb + o_P(\delta)$ if $\|\ba - \bb\| = o_P(\delta)$. We denote $\ba \overset{d} \Rightarrow \mathcal L$ for some distribution $\mathcal L$ if there exists $\bb \sim \mathcal L$ such that $\ba = \bb + o_P(1)$. In the following, $C$ is a generic constant that may differ from line to line.

\section{Asymptotic behavior of empirical eigen-structure} \label{sec3}
\subsection{Asymptotic normality of empirical eigenvalues}
\label{sec3.1}
Let us first study the behavior of the first $m$ empirical eigenvalues of $\hat \bSigma$. Denote by $\lambda_j(\bA)$ the $j^{th}$ largest eigenvalue of matrix $\bA$ and recall that $\hat\lambda_j = \lambda_j (\hat\bSigma)$. We have the following asymptotic normality of $\hat\lambda_j$.

\begin{thm}
\label{thm3.1}
Under Assumptions \ref{assump1} - \ref{assump3}, $\{\hat \lambda_j \}_{j=1}^m$'s have independent limiting distributions.  In addition,
\begin{equation} \label{eq3.1}
    \sqrt{n} \Big\{ \frac{\hat \lambda_j}{\lambda_j} - \Big(1+ \bar c c_j + O_P(\lambda_j^{-1} \sqrt{p/n})\Big) \Big\} \overset{d} \Rightarrow N(0,\kappa_j-1)\,,
\end{equation}
where $\kappa_j$ is the kurtosis of $X_j$.
\end{thm}

The theorem shows that the bias of $\hat \lambda_j/\lambda_j$ is $\bar{c} c_j + O_P(\lambda_j^{-1} \sqrt{p/n})$. The second term is dominated by the first term since $p>n$ and it is of order $o_P(n^{-1/2})$ if $\sqrt{p}= o(\lambda_j)$.  The latter assumption is satisfied by the strong factor model in \cite{FanLiaMin13} and a part of weak factor model in \cite{Ona12}.  To get the asymptotically unbiased estimate, it requires $c_j = p/(n \lambda_j) \to 0$ for $j \le m$.
This result is more general than that of \cite{SheSheZhuMar13} and sheds a similar light to that of \cite{KolLou14a, KolLou14b} i.e. $\|\hat\bSigma - \bSigma\| \to 0$ almost surely if and only if the effective rank $\bar r = \tr(\bSigma)/\lambda_1$ is of order $o(n)$, which is true when $c_1 = o(1)$.
\cite{YatAos12, YatAos13} employed a similar technical trick and gave a comprehensive study on the asymptotic consistency and distributions of the eigenvalues.  They got various similar results  under different conditions from ours.  Our framework is more general and bias reduction can also be made by using a different method; see Section \ref{sec4.2}.
In addition, under the typical spiked covariance model as in \cite{BaiAroPec05}, \cite{JohLu09} and \cite{Pau07}, where it is assumed $\lambda_j = c_0 = C_0, j > m$, we have $\bar c = c_0$ equal to the minimum eigenvalue of the population covariance matrix. The theorem reveals the bias is controlled at the rate $p/(n\lambda_j)$.  Our result is also consistent with \cite{And63}'s result that
\[
\sqrt{n}\Big(\hat \lambda_j - \lambda_j \Big)  \overset{d} \Rightarrow N(0,2\lambda_j^2)\,,
\]
for Gaussian distributions and fixed $p$ and $\lambda_j$'s, where the non-spiked part does not exist and thus the bias $O_P(\lambda_j^{-1} \sqrt{p/n})$ disappears. The proof is relegated to Section \ref{sec6}.

\subsection{Behavior of empirical eigenvectors} \label{sec3.2}
Let us consider the asymptotic distribution of the empirical eigenvectors $\hat \bxi_j$'s corresponding to $\hat\lambda_j$, $j = 1,2,\dots, m$. As in \cite{Pau07}, each $\hat\bxi_j$ is divided into two parts corresponding to the spike and non-spike components, i.e. $ \hat\bxi_j = ( \hat\bxi_{jA}',  \hat\bxi_{jB}')'$ where $ \hat\bxi_{jA}$ is of length $m$.

\begin{thm}
\label{thm3.2}
Under Assumptions \ref{assump1} - \ref{assump3}, we have \\
(i) For the spike part, if $m = 1$,
\begin{equation} \label{eeq3.2}
\frac{2(1+\bar c c_1)}{\bar c c_1} \sqrt{n} \Big(\sqrt{1+\bar c c_1}\; \hat\xi_{1A} - 1 + O_P\Big(\sqrt{\frac{p}{n\lambda_1^2}}\Big) \Big) \overset{d} \Rightarrow N(0, \kappa_1 - 1)\,,
\end{equation}
while if $m > 1$,
\begin{equation} \label{eeq3.3}
    \sqrt{n} \Big( \frac{\hat\bxi_{jA}}{\|\hat\bxi_{jA}\|} - \be_{jA} + O_P\Big(\sqrt{\frac{p}{n\lambda_j^2}}\Big)\Big)  \overset{d} \Rightarrow N_m(\bf 0,\bSigma_j)\,,
\end{equation}
for $j=1,2,\dots,m$, with
\[
\bSigma_j = \sum\nolimits_{ k \in [m] \setminus j} a_{jk}^2 \be_{kA} \be_{kA}'\,,
\]
where $[m]=\{1, \cdots, m\}$, $\be_{kA}$ is the first $m$ elements of unit vector $\be_k$, and $a_{jk} = \lim_{\lambda_j, \lambda_k \to \infty} \sqrt{\lambda_j \lambda_k}/(\lambda_j - \lambda_k)$, which is assumed to exist. \\
(ii) For the noise part, if we further assume the data is Gaussian, there exists $p-m$ dimensional vector $\bh_0$ such that
\begin{equation}  \label{eeq3.5}
\Big\| \bD_0 \frac{\hat\bxi_{jB}}{\|\hat\bxi_{jB}\|} - \bh_0\Big\| = O_P\Big(\sqrt{\frac{n}{p}}\Big) + o_P\Big(\frac{1}{\sqrt{n}}\Big) \;\; \text{and} \;\; \bh_0 \sim \unif\Big(B_{p-m}(1)\Big)\,,
\end{equation}
where $\bD_0 = \diag(\sqrt{\bar c/ \lambda_{m+1}}, \dots, \sqrt{\bar c/\lambda_p})$ is a diagonal scaling matrix and $\unif(B_k( r))$ denotes the uniform distribution over the centered sphere of radius $r$.  In addition, the max norm of $\hat \bxi_{jB}$ satisfies
\begin{equation}  \label{eeq3.6}
\|\hat \bxi_{jB}\|_{\max} = O_P\Big(p/(n \lambda_j^{3/2}) + \sqrt{\log p/(n\lambda_j)} \Big) \,.
\end{equation} \\
(iii) Furthermore,
$\|\hat\bxi_{jA}\| = (1+ \bar c c_j)^{-1/2} + O_P(\lambda_j^{-1}\sqrt{p/n} + p/(n^{3/2}\lambda_j) )$ and \\
$\|\hat\bxi_{jB}\| = (\frac{\bar c c_j}{1+ \bar c c_j})^{1/2} + O_P(\sqrt{1/\lambda_j} + \sqrt{p/(n^{2}\lambda_j)} )$. Together with (i), this implies the inner product between empirical eigenvector and the population one converges to $(1+\bar c c_j)^{-1/2}$ in probability and
\begin{equation} \label{eeq3.4}
\langle \hat\bxi_j, \be_j \rangle  - \frac{1}{\sqrt{1+\bar c c_j}}  =
O_P\Big(\lambda_j^{-1} \sqrt{p/n} + p/(n^{3/2}\lambda_j)\Big) + O_P(n^{-1}) I(m > 1).
\end{equation}
\end{thm}

In the above theory, we assume that $a_{jk} = \lim_{\lambda_j, \lambda_k \to \infty} \frac{\sqrt{\lambda_j \lambda_k}}{\lambda_j - \lambda_k}$ exists. This is not restrictive if eigenvalues are well separated i.e. $\min_{j \ne k \le m} |\lambda_j - \lambda_k|/\lambda_j \ge \delta_0$ from assumption \ref{assump1}.
The assumption obviously holds for the pervasive factor model \citep{FanLiaMin13}, in which
$a_{jk} = \sqrt{\lambda_j(\bSigma_b) \lambda_k(\bSigma_b)}/(\lambda_k(\bSigma_b) -  \lambda_j(\bSigma_b))$.

Theorem~\ref{thm3.2} is an extension of random matrix results into ultra high dimensional regime. Its proof sheds light on how to use the smaller $n \times n$ matrix $\widetilde \bSigma$ as a tool to understand the behavior of the larger $p \times p$ covariance matrix $\hat \bSigma$. Specifically, we start from $\widetilde \bSigma \bu_j = \hat \lambda_j \bu_j$ or identity (\ref{eqn4.3}) and then use the simple fact (\ref{VecRel}) to get a relationship (\ref{eqn4.4}) of eigenvector $\hat\bxi_j$. Then  (\ref{eqn4.4}) is rearranged as (\ref{eqn4.5}) which gives a clear separation of dominating term, that is asymptotically normal, and error term. This makes the whole proof much simpler in comparison with \cite{Pau07} who showed a similar type of results through a complicated representation of $\hat\bxi_j$ and $\hat \lambda_j$. From this simple trick, we can understand deeply how some important high and low dimensional quantities link together and differ from each other.

Several remarks are in order. Firstly, since  $\hat\bxi_j^{(\bY)} = \bGamma \hat\bxi_j$ is the $j^{th}$ empirical eigenvector based on observed data $\bY$, we have decomposition
\[
\hat\bxi_j^{(\bY)} = \bGamma_A \hat\bxi_{jA} + \bGamma_B\hat\bxi_{jB}\,,
\]
where $\bGamma = (\bGamma_A, \bGamma_B)$.  Note that $\bGamma_A \hat\bxi_{jA}$ converges to the true eigenvector deflated by a factor of $\sqrt{1+\bar c c_j}$ with the convergence rate $O_P(\sqrt{p/(n\lambda_j^2)} + p/(n^{3/2} \lambda_j) + n^{-1/2})$ while $\bGamma_B\hat\bxi_{jB}$ creates a random bias, which is distributed uniformly on an ellipse of $(p-m)$ dimension and projected into the $p$ dimensional space spaned by $\bGamma_B$. The two parts intertwined in such a way that correction for the bias of estimating eigenvectors is almost impossible. More details are discussed in Section \ref{sec4} for factor models. Secondly, it is clearly as in the eigenvalue case, the bias term $\lambda_j^{-1}\sqrt{p/n}$ in Theorem \ref{thm3.2} (i) disappears when $\sqrt{p} = o(\lambda_j)$.  In particular, for the stronger factor given by (\ref{eq2.1}), $\hat\bxi_j^{(\bY)}$ is a consistent estimator. Thirdly, the situations $m = 1$ and $m > 1$ have slight difference in that multiple spikes could interact with each other. Especially this reflects in the convergence of angle of empirical eigenvector to its population counterpart: the angle converges to $(1+\bar c c_j)^{-1/2}$ with an extra rate $O_P(1/n)$ which stems from estimating $\hat\xi_{jk}$ for $j \ne k \le m$ (see proof of Theorem \ref{thm3.2} (iii)). The difference will only be seen when the spike magnitude is higher than the order $\sqrt{pn} \vee pn^{-1/2}$. We will verify this by a simple simulation in Section \ref{sec5}. Fourthly, it is the first time that the max norm bound of the non-spiked part was derived. This bound will be useful for analyzing factor models in Section \ref{sec4}.

Theorem~\ref{thm3.2} again implies the results of \cite{SheSheZhuMar13}. It also generalizes the asymptotic distribution of non-spiked part from pure orthogonal invariant case of \cite{Pau07} to more general bounded setting.  In particular, when $p/n \to \infty$, the asymptotic distribution of the normalized non-spiked component is not uniform over a sphere any more, but over an ellipse. In addition, our result can be compared with the low dimensional case, where \cite{And63} showed that
\beq
    \sqrt{n} \Big( \hat\bxi_{j} - \be_{j} \Big) \overset{d} \Rightarrow N_p\Big({\bf 0},\sum\limits_{ k \in [m] \setminus j } \frac{\lambda_j \lambda_k}{(\lambda_j-\lambda_k)^2} \be_{k} \be_{k}'\Big)\,,
\eeq
for fixed $p$ and $\lambda_j$'s. Under our assumptions, if the spiked eigenvalues go to infinity, the constants in the asymptotic covariance matrix are replaced by the limits $a_{jk}$'s. Similar to the behavior of eigenvalues, the spiked part $\hat\bxi_{jA}$ preserves the normality property except for a bias factor $1/(1+\bar c c_j)$ caused by the high dimensionality.
 
Recent manuscript by \cite{KolLou14b} provides general asymptotic results for the empirical eigenvectors from a spectral projector point of view, but they mainly focus on the regime of $p/n\lambda_j \to 0$.  Indeed, they limit themselves to the regime that $p=o(n)$ and $\lambda_1 = O(1)$ when establishing the asymptotic normality (see conditions for Theorems 5 and 7 therein).  In contrast, we consider a very different regime, requiring $p > n$ and allowing $\lambda_1$ to diverge.  Furthermore, Theorem~\ref{thm3.2}
gives a more refined description on the behavior of empirical eigenvectors than the asymptotic normality result given in Theorem 7 of \cite{KolLou14b}.
Last but not least, it has been shown by \cite{JohLu09} that PCA generates consistent eigenvector estimation if and only if $p/n \to 0$ when the spike sizes are fixed. This motivates the study of sparse PCA. We take the spike magnitude of eigenvalues into account and provide additional insights by showing that PCA consistently estimate eigenvalues and eigenvectors if and only if $p/(n\lambda_j) \to 0$. This explains why \cite{FanLiaMin13} can consistently estimate the eigenvalues and eigenvectors while \cite{JohLu09} can not.

\section{Applications to factor models} \label{sec4}
In this section, we propose a method named Shrinkage Principal Orthogonal complEment Thresholding (S-POET) for estimating large covariance matrices induced by the approximate factor models. The estimator is based on correction of the bias of empirical eigenvalues as specified in (\ref{eq3.1}). We derive for the first time the bound of the relative estimation errors of covariance matrices under the spectral norm. The results are then applied to assessing large portfolio risk and estimation of false discovery proportion, where the conditions in existing literature are relaxed.

\subsection{Approximate factor models}
Factor models have been widely used in various disciplines such as finance and genomics.  Consider the approximate factor model
\beq \label{eeq4.1}
y_{it} = \bb_i' \bff_t + u_{it} \,,
\eeq
where $y_{it}$ is the observed data for the $i^{th}$ ($i=1,\dots,p$) individual (e.g. returns of stocks) or components (e.g. expression of genes) at time $t=1,\dots,T$; $\bff_t$ is a $m \times 1$ vector of latent common factors and $\bb_i$ is the factor loadings for the $i^{th}$ individuals or components; $u_{it}$ is the idiosyncratic error, uncorrelated with the common factors. In genomics application, $t$ can also index individuals or repeated experiments.   For simplicity we assume there is no time dependency.

The factor model can be written into a matrix form as follows:
\beq
\bY = \bB \bF' +\bU\,,
\eeq
where $\bY_{p\times T}$, $\bB_{p\times m}$, $\bF_{T \times m}$, $\bU_{p\times T}$ are respectively the matrix form of observed data, factor loading matrix, factor matrix, and error matrix.  For identifiability issue, we impose the condition that $\cov(\bff_t) = \bI$ and $\bB'\bB$ is a diagonal matrix. Thus, the covariance matrix is given by
\beq
\bSigma = \bB\bB' + \bSigma_u\,,
\eeq
where $\bSigma_u$ is the covariance matrix of the idiosyncratic error at any time $t$.

Under the assumption that $\bSigma_u = (\sigma_{u,ij})_{i,j \le p}$ is sparse with its eigenvalues bounded away from zero and infinity, the population covariance exhibit a ``low-rank plus sparse'' structure. The sparsity is measured by the following quantity
\[
m_p = \max_{i \le p} \sum_{j \le p} |\sigma_{u,ij}|^q,
\]
for some $q \in [0,1]$ \citep{BicLev08}. In particular, $m_p$ with $q=0$ is the maximum number of nonzero elements in each row of $\bSigma_u$.

In order to estimate the true covariance matrix with the above factor structure, \cite{FanLiaMin13} proposed a method called ``POET" to recover the unknown factor matrix as well as the factor loadings. The idea is simply to first decompose the sample covariance matrix into the spiked part and non-spiked part and estimate them separately.  Specifically, let $\hat\bSigma = T^{-1} \bY \bY'$ and $\{\hat\lambda_j\}$ and $\{\hat \bxi_j\}$ be its corresponding eigenvalues and eigenvectors.  They define
\beq
\label{Eqn:decomp}
\hat\bSigma^{\top} = \sum_{j = 1}^{m} \hat\lambda_j \hat\bxi_{j} \hat\bxi_j' + \hat\bSigma_u^{\top} \,,
\eeq
where $\hat\bSigma_u^{\top}$ is the matrix after applying thresholding method
\citep{BicLev08} to $\hat \bSigma_u = \hat\bSigma - \sum_{j=1}^m \hat\lambda_j \hat\bxi_{j} \hat\bxi_j'$.

They showed that the above estimation procedure is equivalent to the least square approach that minimizes
\beq \label{minimization}
(\hat\bB, \hat\bF) = \arg \min_{\bB,\bF} \|\bY - \bB\bF'\|_F^2 \text{ s.t. } \frac{1}{T} \bF'\bF = \bI_m, \bB'\bB \text{ is diagonal}.
\eeq
The columns of $\hat\bF/\sqrt{T}$ are the eigenvectors corresponding to the $m$ largest eigenvalues of the $T\times T$ matrix $T^{-1} \bY'\bY$ and $\hat\bB = T^{-1} \bY \hat\bF$. 
After $\bB$ and $\bF$ are estimated, the sample covariance of $\hat \bU = \bY - \hat\bB \hat\bF'$ can be formed: $\hat\bSigma_u = T^{-1} \hat\bU \hat\bU'$. Finally thresholding is applied to $\hat\bSigma_u$ to generate $\hat\bSigma_u^{\top} = (\hat\sigma_{u,ij}^{\top})_{p\times p}$, where
\begin{equation} \label{thresholding}
\hat\sigma_{u,ij}^{\top} = \left\{  \begin{array}{lr} \hat\sigma_{u,ij}, & i = j; \\ s_{ij} (\hat\sigma_{u,ij}) I(|\hat\sigma_{u,ij}| \ge \tau_{ij}),  & i \ne j. \end{array} \right.
\end{equation}
Here $s_{ij}(\cdot)$ is the generalized shrinkage function \citep{AntFan01,RotLevZhu09} and $\tau_{ij} = \tau (\hat\sigma_{u,ii} \hat\sigma_{u,jj})^{1/2}$ is the entry-dependent threshold. The above adaptive threshold corresponds to applying thresholding with parameter $\tau$ to the correlation matrix of $\hat\bSigma_u$. The positive parameter $\tau$ will be determined later.


\cite{FanLiaMin13} showed that under Assumptions \ref{assPoet1} - \ref{assPoet4} listed in Appendix \ref{secA} in the supplementary material \citep{EigenStructure_supp},
\beq \label{Eqn:oldRate}
\|\hat\bSigma^{\top} - \bSigma\|_{\Sigma, F} = O_P\Big(\frac{\sqrt{p}\log p}{T} + m_p\Big(\frac{\log p}{T} + \frac{1}{p}\Big)^{(1-q)/2}\Big)\,,
\eeq
where $\|\bA\|_{\Sigma, F} = p^{-1/2} \|\bSigma^{-1/2}\bA \bSigma^{-1/2}\|_F$ and $\|\cdot\|_F$ is the Frobenius norm. Note that
$$
    \|\hat\bSigma^{\top} - \bSigma\|_{\Sigma, F} = p^{-1/2}
    \|\bSigma^{-1/2} \hat\bSigma^{\top} \bSigma^{-1/2} - \bI_p \|_F,
$$
which measures the relative error in Frobenius norm. A more natural metric is relative error under the operator norm $\|\bA\|_{\Sigma} = p^{-1/2} \|\bSigma^{-1/2}\bA \bSigma^{-1/2}\|$, which can not be obtained by using the
technical device of \cite{FanLiaMin13}.  Via our new tools, we will establish such a result under weaker conditions than their pervasiveness assumption. Note that the relative error convergence is particularly meaningful for spiked covariance matrix, as eigenvalues are in different scales.

\subsection{Shrinkage POET under relative spectral norm} \label{sec4.2}

The discussion above reveals several drawbacks of POET. First, the spike size has to be of order $p$ which rules out relatively weaker factors.  Second, it is well known that the empirical eigenvalues are inconsistent if the spike eigenvalues do not significantly dominate the non-spike part. Therefore, proper correction or shrinkage is needed. See a recent paper by \cite{DonGavJoh14} for optimal shrinkage of eigenvalues.

Regarding to the first drawback, we relax the assumption $\|p^{-1}\bB' \bB - \bOmega_0\| = o(1)$ in Assumption \ref{assPoet1} to the following weaker assumption.

\begin{assum} \label{assump1:factor}
$\|\bLambda_A^{-1/2} \bB' \bB \bLambda_A^{-1/2} - \bOmega_0\| = o(1)$ for some $\bOmega_0$ with eigenvalues bounded from above and below, where $\bLambda_A = \diag(\lambda_1, \dots, \lambda_m)$. In addition, we assume $\lambda_m \to \infty$, $\lambda_1/\lambda_m$ is bounded from above and below.
\end{assum}

This assumption does not require the first $m$ eigenvalues of $\bSigma$ to take on any specific rate.  They can still be much smaller than $p$, although for simplicity we require them to diverge and share the same diverging rate. As we assume bounded $\| \bSigma_u\|$, the assumption
$\lambda_m \to \infty$ is also imposed to avoid the issue of identifiability. When $\lambda_m$ does not diverge, more sophisticated condition is needed for identifiability \citep{Cha11}.

In order to handle the second drawback, we propose the Shrinkage POET (S-POET) method. Inspired by (\ref{eq3.1}), the shrinkage POET modifies the first part in POET estimator (\ref{Eqn:decomp}) as follows:
\beq \label{Eqn:SPOET}
\hat\bSigma^{S} = \sum_{j = 1}^{m} \hat\lambda_j^S \hat\bxi_{j} \hat\bxi_j' + \hat\bSigma_u^{\top} \,,
\eeq
where $\hat\lambda_j^S = \max\{\hat\lambda_j - \bar c p/n, 0\}$, a simple soft thresholding correction. Obviously if $\hat\lambda_j$ is sufficiently large, $\hat\lambda_j^S/\lambda_j = \hat\lambda_j/\lambda_j - \bar c c_j = 1 + o_P(1)$. Since $\bar c$ is unknown, a natural estimator $\hat c$ is such that the total of the eigenvalues remains unchanged:
\[
\tr(\hat\bSigma) = \sum_{j=1}^m (\hat\lambda_j - \hat c p/n) + (p-m) \hat c\,
\]
or $\hat c = (\tr(\hat\bSigma) - \sum_{j=1}^m \hat\lambda_j)/(p-m-pm/n)$. It has been shown by Lemma 7 of \cite{YatAos12} that
\[
(\hat c - \bar c) \frac{p}{n\lambda_j} = O_P\Big(\frac{\tr(\hat\bSigma) - \sum_{j=1}^m \hat\lambda_j}{(n-m)\lambda_m} - \frac{\bar c p}{n\lambda_m} \Big) = O_P(n^{-1})\,.
\]
Thus, replacing $\bar c$ by $\hat c$, we have $\hat\lambda_j^S/\lambda_j - 1 = O_P(\lambda_j^{-1}\sqrt{p/n} + n^{-1/2})$, i.e. the estimation error in $\hat{c}$ is negligible. From Theorem \ref{thm3.1}, we can easily obtain asymptotic normality, that is $\sqrt{n}(\hat\lambda_j^S/\lambda_j - 1) \overset{d} \Rightarrow N(0,\kappa_j-1)$ if $\sqrt{p} = o(\lambda_j)$.

To get the convergence of relative errors under the operator norm, we also need the following additional assumptions:

\begin{assum}  \label{assump2:factor}
(i) $\{\bu_t, \bff_t\}_{t \ge 1}$ are independently and identically distributed with $\mathbb E[u_{it}] = \mathbb E[u_{it} f_{jt}] = 0$ for all $i \le p, j \le m$ and $t \le T$. \\
(ii) There exist positive constants $c_1$ and $c_2$ such that $\lambda_{\min}(\bSigma_u) > c_1$, $\|\bSigma_u\|_{\infty} < c_2$, and $\min_{i,j} \Var(u_{it} u_{jt}) > c_1$. \\
(iii) There exist positive constants $r_1, r_2, b_1$ and $b_2$ such that for $s>0, i\le p, j \le m$,
$$
\mathbb P(|u_{it}| > s) \le exp(-(s/b_1)^{r_1}) \; \text{ and } \; \mathbb P(|f_{jt}| > s) \le exp(-(s/b_2)^{r_2})\,.
$$
(iv) There exists $M > 0$ such that for all $i \le p, j \le m$, $|b_{ij}| \le M\sqrt{\lambda_j/p}$. \\
(v) $\sqrt{p} (\log T)^{1/r_2} = o(\lambda_m)$.
\end{assum}


The first three conditions are common in factor model literature. If we write $\bB = (\widetilde\bb_1, \dots, \widetilde\bb_m)$, by Weyl's inequality we have $\max_{1\le j\le m} \|\widetilde\bb_j\|^2/\lambda_j \le 1+\|\bSigma_u\|/\lambda_j = 1+o(1)$. Thus it is reasonable to assume the magnitude $|b_{ij}|$ of factor loadings is of order $\sqrt{\lambda_j/p}$ in the fourth condition.
The last condition is imposed to ease technical presentation.

Now we are ready to investigate $\|\hat\bSigma^S - \bSigma\|_{\Sigma}$. Suppose the SVD decomposition of $\bSigma$,
$$
\bSigma = \left( \begin{array}{cc}
\bGamma_{p \times m} & \bOmega_{p \times (p-m)} \end{array} \right) \left( \begin{array}{cc}
\bLambda_{m \times m} & \\ & \bTheta_{(p-m) \times (p-m)} \end{array} \right)   \left( \begin{array}{c} \bGamma' \\ \bOmega' \end{array} \right)\,.
$$
Then obviously
\beqn \label{Norm_BreakDown1}
\|\hat\bSigma^S - \bSigma\|_{\Sigma} & \le & \left\| \bSigma^{-\frac12} (\hat\bGamma \hat\bLambda^S\hat\bGamma' - \bB \bB') \bSigma^{-\frac12} \right\| + \|\bSigma^{-\frac12} (\hat\bSigma_u^{\top} - \bSigma_u) \bSigma^{-\frac12}\| \nonumber \\
 & =: & \Delta_L + \Delta_S,
\eeqn
and
\beq \label{Norm_BreakDown3}
\Delta_S \le \|\bSigma^{-1}\| \|\hat\bSigma_u^{\top} - \bSigma_u\| \le C \|\hat\bSigma_u^{\top} - \bSigma_u\| \,.
\eeq
It can be shown
\beq \label{Norm_BreakDown2}
\begin{aligned}
\Delta_L &=  \left\| \left( \begin{array}{c} \bLambda^{-\frac12} \bGamma' \\ \bTheta^{-\frac12} \bOmega' \end{array} \right) (\hat\bGamma \hat\bLambda^S \hat\bGamma' - \bB \bB') \left( \begin{array}{cc}
\bGamma \bLambda^{-\frac12} & \bOmega \bTheta^{-\frac12} \end{array} \right) \right\|  \\
& \le \Delta_{L1} + \Delta_{L2} \,,
\end{aligned}
\eeq
where $\Delta_{L1} =  \|\bLambda^{-\frac12} \bGamma' (\hat\bGamma\hat\bLambda^S \hat\bGamma' - \bB \bB') \bGamma \bLambda^{-\frac12}\|$
and $\Delta_{L2} = \|\bTheta^{-\frac12} \bOmega' (\hat\bGamma \hat\bLambda^S \hat\bGamma' - \bB \bB') \bOmega \bTheta^{-\frac12} \|$.
Thus in order to find the convergence rate of relative spectral norm, we need to consider the terms $\Delta_{L1}, \Delta_{L2}$ and $\Delta_S$ separately. Notice that $\Delta_{L1}$ measures the relative error of the estimated spiked eigenvalues, $\Delta_{L2}$ reflects the goodness of the estimated eigenvectors, and $\Delta_S$ controls the error of estimating the sparse idiosyncratic covariance matrix. The following theorem reveals the rate of each term. Its proof will be provided in Appendix \ref{secB} of the supplementary material \citep{EigenStructure_supp}.

\begin{thm}  \label{thm_RelSpectral}
Under Assumptions \ref{assump1}, \ref{assump2}, \ref{assump3}, \ref{assump1:factor} and \ref{assump2:factor}, if $p\log p > \max\{ T(\log T)^{4/r_2}, T(\log (pT))^{2/r_1} \}$, we have
$$
\Delta_{L1} = O_P\Big(T^{-1/2}\Big)\,, \;\; \Delta_{L2} = O_P\Big(\frac{p}{T} + \frac{1}{\lambda_m} \Big)\,,
$$
and by applying adaptive thresholding estimator (\ref{thresholding})
with
$$
\tau_{ij} = C \omega_T (\hat\sigma_{u,ii} \hat\sigma_{u,jj})^{1/2}, \quad
\mbox{and} \quad \omega_T = \sqrt{\log p/T} + \sqrt{1/p},
$$
we have
$$
\Delta_S = O_P\Big(m_p \omega_T^{1-q} \Big)\,.
$$
Combining the three terms, $\|\hat\bSigma^S - \bSigma\|_{\Sigma} = O_P(T^{-1/2} + \frac{p}{T} + \frac{1}{\lambda_m} +m_p \omega_T^{1-q} )$.
\end{thm}

The relative error convergence in spectral norm characterizes the accuracy of estimation for spiked covariance matrix.  In contrast with the previous results on Frobenius or max norm, this is the first time that the relative rate under spectral norm is derived. When $\lambda_m \asymp p$ and $q = 0$, we have
$$
\|\hat\bSigma^S - \bSigma\|_{\Sigma} = O_P\Big( \frac{p}{T} + m_p \sqrt{\frac{\log p}{T}} + \sqrt{\frac 1 p}\Big)\,.
$$
Comparing the rate with (\ref{Eqn:oldRate}), we see the difference under two different norms. The term $\sqrt{p}\log p/T$ in (\ref{Eqn:oldRate}) is enlarged to rate $p/T$, which is due to the incoherence of the eigen-spaces of the low-rank signal matrix and sparse error matrix. Specifically this rate comes from $\Delta_{L2}$. If we care only the relative error of the low-rank and sparse matrix spaces separately, we should only emphasize on $\Delta_{L1}$ and $\Delta_S$.

If $c_j = o(1)$, the proposed $\hat\lambda_j^S$ is asymptotically just the spiked empirical eigenvalue $\hat\lambda_j$.  However, when we have semi-weak factors whose corresponding eigenvalues are as weak as $p/T$, shrinkage is necessary to guarantee the convergence of $\Delta_{L1}$. On the other hand, if instead POET is applied to estimate covariance matrix, $\Delta_{L1} = O_P(p/(\lambda_m T) +T^{-1/2})$ which is only bounded. However since the empirical eigenvectors are not corrected, POET and S-POET attain the same rate for $\Delta_{L2}$, which actually dominates $\Delta_{L1}$ and $\Delta_S$ in high dimensional setting. Nevertheless, as to be seen in the simulation studies,
S-POET can stabilize the estimator and improve the estimation accuracy. For this reason, we recommend S-POET in practice.


\subsection{Portfolio risk management}
Portfolio allocation and risk management have been a fundamental problem in finance since \cite{Mar52}'s groundbreaking work on minimizing the volatility of portfolios with a given expected return. Specifically, the risk of a given portfolio with allocation vector $\bw$ is conventionally measured by its variance $\bw'\bSigma\bw$, where $\bSigma$ is the volatility (covariance) matrix of the returns of underlying assets. To estimate large portfolio's risks, it needs to estimate a large covariance matrix $\bSigma$ and factor models are frequently used to reduce the dimensionality. This was the idea of \cite{FanLiaShi15} in which they used POET estimator to estimate $\bSigma$. However, the basic method for bounding the risk error $|\bw'\hat\bSigma\bw-\bw'\bSigma\bw|$ in their paper as well as another earlier paper of similar topic \citep{FanZhaYu12} was
$$
|\bw'\hat\bSigma\bw-\bw'\bSigma\bw| \le \|\bw\|_1^2 \|\hat\bSigma - \bSigma\|_{\max}\,.
$$
They assumed that the gross exposure of the portfolio is bounded, mathematically $\|\bw\|_1 = O(1)$, which made it possible to only focus on the max error norm.
Technically, when $p$ is large, $\bw'\bSigma\bw$ can be small.
What an investor cares mostly is the relative risk error $\RE(\bw) = |\bw'\hat\bSigma\bw/\bw'\bSigma\bw - 1|$. Often $\bw$ is a data-driven investment strategy, which is a random variable itself.  Regardless of what $\bw$ is,
$$
\max_{\bw} \RE(\bw)   = \|\hat\bSigma - \bSigma\|_{\bSigma},
$$
which does not converge by Theorem \ref{thm_RelSpectral}.  The question is what kind of portfolio $\bw$ will make the relative error converge. Decompose $\bw$ as a linear combination of the eigenvectors of $\bSigma$, namely $\bw = (\bGamma, \bOmega) \bfeta$ and $\bfeta = (\bfeta_A', \bfeta_B')'$. We have the following useful result for risk management.

\begin{thm}  \label{thm_Risk}
Under Assumptions \ref{assump1}, \ref{assump2}, \ref{assump1:factor},\ref{assump2:factor} and the factor model (\ref{eeq4.1}) with Gaussian noises and factors, if there exists $C_1 >0$ such that $\|\bfeta_B\|_1 \le C_1$, and assume $\lambda_j \propto p^{\alpha}$ for $j = 1,\dots,m$ and $T \ge Cp^{\beta}$ for $\alpha >  1/2, 0 < \beta < 1, \alpha + \beta > 1$, then the relative risk error is of order
$$
\RE(\bw) = \Big|\frac{\bw'\hat\bSigma^S\bw}{\bw'\bSigma\bw}-1\Big| = O_P\Big( T^{-\min \{\frac{2(\alpha+\beta-1)}{\beta}, \frac12\}} + m_p w_T^{1-q}\Big)\,,
$$
for $\alpha < 1$. If $\alpha \ge 1$ or $\|\bfeta_A\| \ge C_2$, $\RE(\bw) =  O_P( T^{-1/2} + m_p w_T^{1-q})$.
\end{thm}

The condition $\|\bfeta_B\|_1 \le C_1$ is obviously much weaker than $\|\bw\|_1 = O(1)$. It does not limit the total exposure of investor's position, but only put constraint on investment of the non-spiked section. Note that under the conditions of Theorem~\ref{thm_Risk}, $p/(T\lambda_j)  \to 0$, and S-POET and POET are approximately the same. The stated result hold for POET too.

\subsection{Estimation of false discovery proportion}
Another important application of the factor model is the estimation of false discovery proportion.  For simplicity, we  assume Gaussian data $\bX_i \sim N(\bmu,\bSigma)$ with an unknown correlation matrix $\bSigma$ and wish to test separately which coordinates of $\bmu$ are nonvanishing. Consider the test statistic $\bZ = \sqrt{n}\bar\bX$ where $\bar \bX$ is the sample mean of all data. Then $\bZ \sim N(\bmu^*, \bSigma)$ with $\bmu^*=\sqrt{n}\bmu$ and the problem is to test
$$
H_{0j}: \mu_j^* = 0  \quad\quad \text{vs} \quad\quad H_{1j}: \mu_j^* \ne 0.
$$
Define the number of discoveries $R(t) = \#\{j: P_j \le t\}$ and the number of false discoveries $V(t) = \#\{\mbox{true null}: P_j \le t\}$, where $P_j$ is the p-value associated with the $j^{th}$ test. Note that $R(t)$ is observable while $V(t)$ needs to be estimated. The false discovery proportion (FDP) is defined as $\FDP(t) = V(t)/R(t)$.

Recently \cite{FanHan13} proposed to employ the factor structure
\beq \label{FDP_model}
\bSigma = \bB\bB' + \bA\,,
\eeq
where $\bB = (\sqrt{\lambda_1}\bxi_1, \dots,\sqrt{\lambda_m}\bxi_m)$ and $\lambda_j$ and $\bxi_j$ are respectively the $j^{th}$ eigenvalue and eigenvector of $\bSigma$ as before. Then $\bZ$ can be stochastically decomposed as
$$
\bZ = \bmu^* + \bB \bW +\bK\,,
$$
where $\bW \sim N({\bf 0}, \bI_m)$ are $m$ common factors and $\bK \sim N({\bf 0}, \bA)$ independent of $\bW$ are the idiosyncratic errors. For simplicity, assume the maximal
number of nonzero elements of each row of $\bA$ is bounded. In \cite{FanHan13}, they demonstrated that a good approximation for $\FDP(t)$ is
\beq \label{FDP_A}
\FDP_A(t) = \sum_{i=1}^p[\Phi(a_i(z_{t/2}+\eta_i))+\Phi(a_i(z_{t/2}-\eta_i))]/R(t)\,,
\eeq
where $z_{t/2}$ is the $t/2$-quantile of the standard normal distribution, $a_i = (1-\|\bb_i\|^2)^{-1/2}$, $\eta_i = \bb_i' \bW$ and $\bb_i'$ is the $i^{th}$ row of $\bB$.

Realized factors $\bW$ and the loading matrix $\bB$ are typically unknown. If a generic estimator $\hat\bSigma$ is provided, then we are able to estimate $\bB$ and thus $\bb_i$ from its empirical eigenvalues and eigenvectors $\hat\lambda_j$'s and $\hat\bxi_j$'s. $\bW$ can be estimated by the least-squares estimate $\hat\bW = (\hat\bB' \hat\bB)^{-1} \hat\bB' \bZ$. \cite{FanHan13} proposed the following estimator for $\FDP_A(t)$:
\beq \label{FDP_U}
\widehat {\FDP}_U(t) = \sum_{i=1}^p [\Phi(\hat a_i(z_{t/2}+\hat\eta_i))+\Phi(\hat a_i(z_{t/2}-\hat\eta_i))]/R(t)\,,
\eeq
where $\hat a_i = (1-\|\hat \bb_i\|^2)^{-1/2}$ and $\hat \eta_i = \hat\bb_i' \hat \bW$. The following assumptions are in their paper.

\begin{assum} \label{assump1:FDP}
There exists a constant $h > 0$ such that (i) $R(t)/p > h p^{-\theta}$ for $h > 0$ and $\theta \ge 0$ as $p\to \infty$ and (ii) $\hat a_i \le h, a_i \le h$ for all $i = 1, \dots, p$.
\end{assum}

They showed that if $\hat\bSigma$ is based on the POET estimator with a spike size $\lambda_m \asymp p$, under Assumptions \ref{assPoet1} - \ref{assPoet4} together with Assumption~\ref{assump1:FDP},
\beq \label{Eqn:FDP_oldrate}
|\widehat{\FDP}_{U, POET}(t) - \FDP_A(t) | = O_P\Big(p^{\theta} \Big(\sqrt{\frac{\log p}{T}} + \frac{\|\bmu^*\|}{\sqrt{p}}\Big)\Big)\,.
\eeq
Again we can relax the assumption of spike magnitude from order $p$ to much weaker Assumption \ref{assump1:factor}. Since $\bSigma$ is  a correlation matrix,  $\lambda_1 \le \tr(\bSigma) = p$.
This, together with Assumption \ref{assump1:factor}, leads us to consider that all leading eigenvalues are of order proportional to $p^{\alpha}$ for $1/2 < \alpha \le 1$.

Now apply the proposed S-POET method to obtain $\hat\bSigma^S$ and use it for FDP estimation. Then we have the following theorem.

\begin{thm}  \label{thm_FDP}
If Assumptions \ref{assump1}, \ref{assump2}, \ref{assump1:factor}, \ref{assump2:factor}, \ref{assump1:FDP} are applied to Gaussian independent data $\bX_i \sim N(\bmu,\bSigma)$, and $\lambda_j \propto p^{\alpha}$ for $j = 1,\dots,m$, $T \ge C p^{\beta}$ for $1/2 < \alpha \le  1, 0 \le \beta < 1, \alpha + \beta > 1$, we have
$$
|\widehat{\FDP}_{U, SPOET}(t) - \FDP_A(t) | = O_P\Big(p^{\theta}(\|\bmu^*\|p^{-\frac12} + T^{-\min\{\frac{\alpha+\beta-1}{\beta}, \frac12\}})\Big)\,.
$$
\end{thm}

Comparing the result with (\ref{Eqn:FDP_oldrate}), this convergence rate attained by S-POET is more general than the rate achieved before. The only difference is the second term, which is $O(T^{-1/2})$ if $\alpha +\frac12 \beta \ge 1$ and $T^{-(\alpha+\beta-1)/\beta}$ if $\alpha +\frac12 \beta < 1$. So we relax the condition from $\alpha = 1$ in \cite{FanHan13} to $\alpha \in (1/2, 1]$. This means a weaker signal than order $p$ is actually allowed to obtain a consistent estimate of false discovery proportion.

\section{Simulations} \label{sec5}

We conducted some simulations to demonstrate the finite sample behaviors of empirical eigen-structure, the performance of S-POET, and validity of applying it to estimate false discovery proportion.

\subsection{Eigen-structure}
In this simulation, we set $n=50$, $p = 500$ and $\bSigma = \diag(50,20,10,1,\dots,1)$, which has three spikes ($m = 3$) $\lambda_1 = 50, \lambda_2 = 20, \lambda_3 = 10$ and corresponding $c_1 = 0.2, c_2 = 0.5, c_3 = 1$.  Data was generated from multivariate Gaussian. The number of simulations is $1000$. The histograms of the standardized empirical eigenvalues $\sqrt{n/2}(\hat \lambda_j/\lambda_j - 1 - c_j)$, and their associated asymptotic distributions (standard normal) are plotted in Figure \ref{fig:eigenVa}. The approximations are very good even for this low sample size $n = 50$.

\begin{figure}
        \centering
       \includegraphics[width=\textwidth]{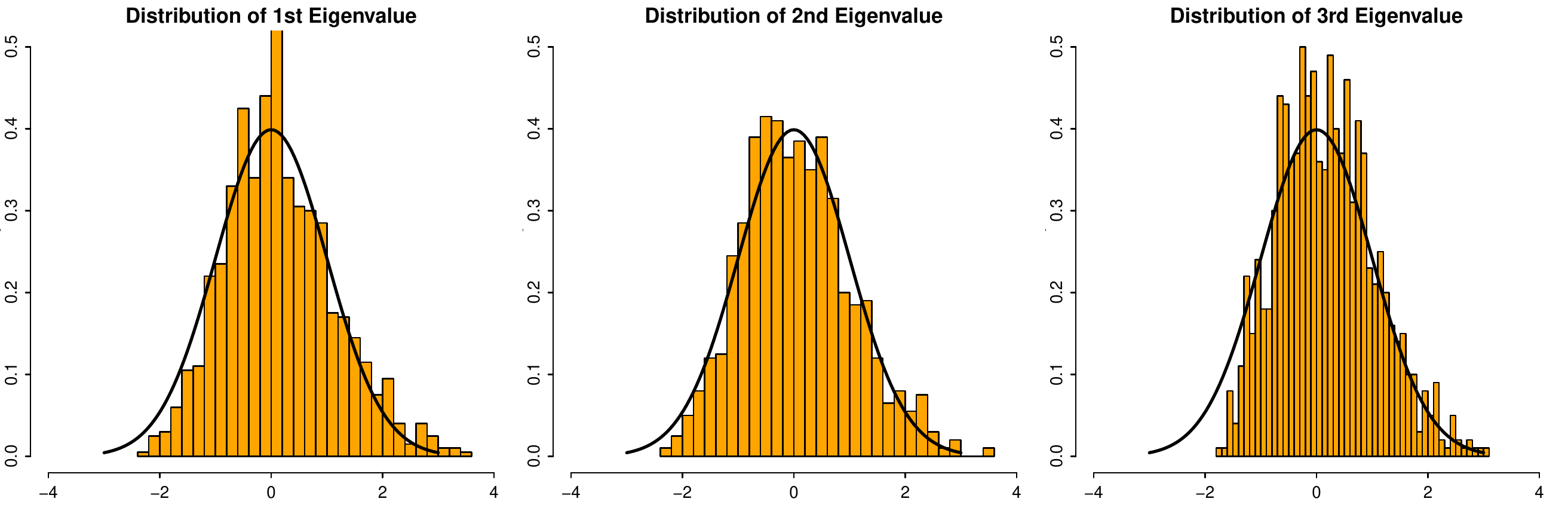}
        \caption{Behaviors of empirical eigenvalues. The empirical distributions of $\sqrt{n/2}(\hat \lambda_j/\lambda_j - 1 - c_j)$ for $j = 1, 2, 3$ are compared with their asymptotic distributions $N(0,1)$.}\label{fig:eigenVa}
\end{figure}

Figure \ref{fig:eigenVe} shows the histograms of $\sqrt{n}(\hat\bxi_{jA}/\|\hat\bxi_{jA}\| - \be_{jA})$ for the first three elements (the spiked part) of the first three eigenvectors. According to the asymptotic result, the values in the diagonal position should stochastically converge to $0$ as observed. On the other hand, plots in the off-diagonal position should converge in distribution to $N(0, 1)$ for $k \ne j$ after standardization, which is indeed the case. We also report the correlations between the first three elements for the three eigenvectors based on those $1000$ repetitions in Table \ref{tab:corr}. The correlations are all quite close to $0$, which is consistent with the theory.

\begin{figure}
        \centering
        \includegraphics[width=\textwidth]{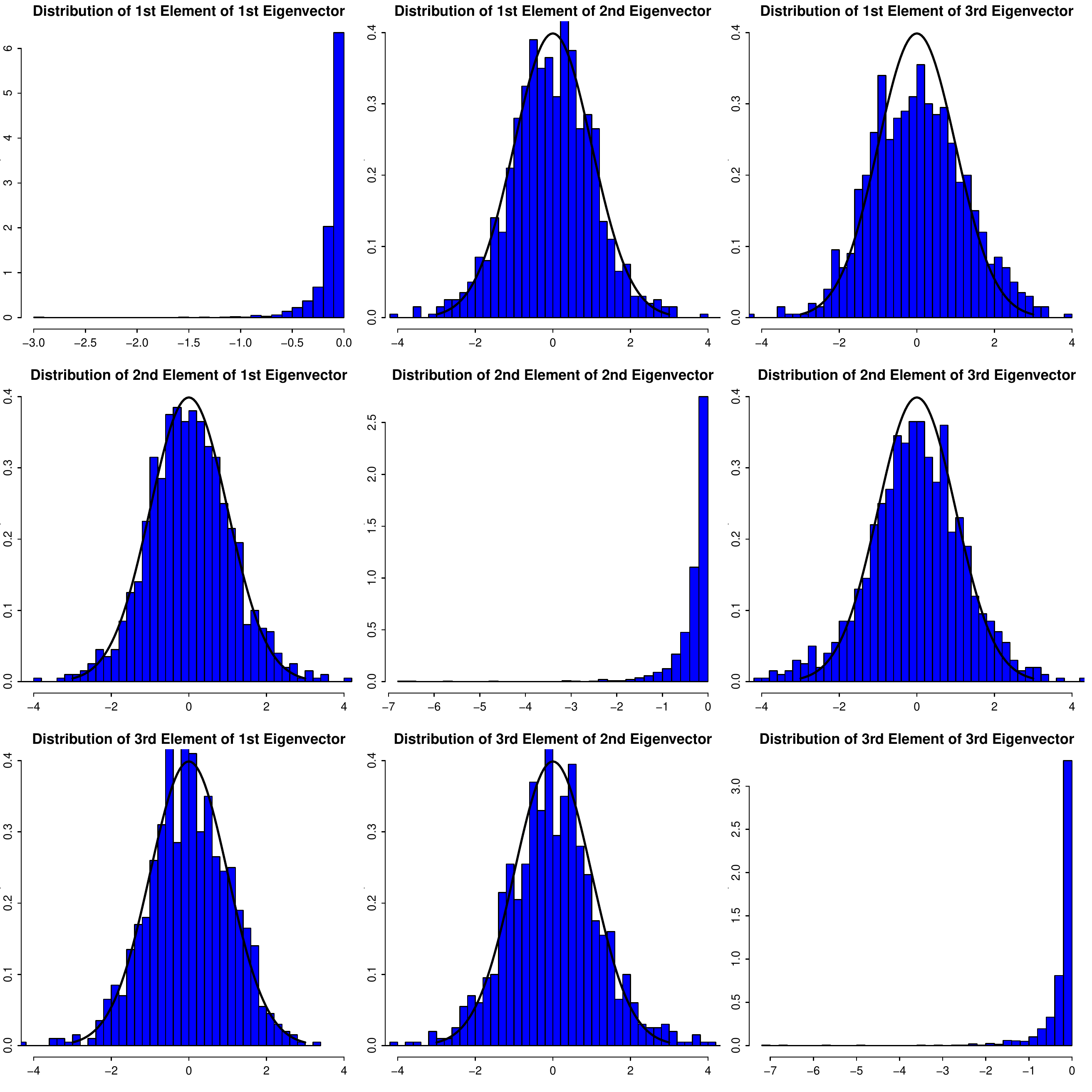}
        \caption{Behaviors of empirical eigenvectors. The histograms of the $k^{th}$ elements of the $j^{th}$ empirical vectors are depicted in the location $(k,j)$ for $k,j \le 3$. Off-diagonal plots of values $\sqrt{n} \hat\xi_{jk}/\|\hat\bxi_{jA}\|/\sqrt{\frac{c_j c_k}{(c_j - c_k)^2}}$ are compared to their asymptotic distributions $N(0, 1)$ for $k \ne j$ while diagonal plots of values $\sqrt{n}(\hat\xi_{jj}/\|\hat\bxi_{jA}\| - 1)$ are compared to stochastically $0$.}\label{fig:eigenVe}
\end{figure}

For the normalized nonspiked part $\hat\bxi_{jB}/\|\hat\bxi_{jB}\|$, it should be distributed uniformly over the unit sphere. This can be tested by the results of \cite{CaiFanJia13}. For any  $n$ data points $\bX_1, \dots, \bX_n$  on $p$-dimensional sphere, define the normalized empirical distribution of angles of each pair of vectors as
$$
\mu_{n,p} = \frac{1}{\binom{n}{2}} \sum_{1\le i < j \le n} \delta_{\sqrt{p-2}(\pi/2 - \Theta_{ij})}\,,
$$
where $\Theta_{ij} \in [0, \pi]$ is the angle between vectors $\bX_i$ and $\bX_j$. When the data are generated uniformly from a sphere, $\mu_{n,p}$ converges to the standard normal distribution with probability $1$.
Figure \ref{fig:eigenAn} shows the empirical distributions of all pairwise angles of the realized $\hat\bxi_{jB}/\|\hat\bxi_{jB}\|$ $(j=1, 2, 3)$ in 1000 simulations. Since number of such pairwise angels is ${1000 \choose 2}$, the empirical distributions and the asymptotic distributions $N(0,1)$ are almost identical. The normality holds even for a small subset of the angles.

\begin{table}
  \centering
    \caption{The correlations between the first three elements for each of the three empirical eigenvectors based on $1000$ repetitions}  \label{tab:corr}
  \begin{tabular}{l|ccc}
    \hline
    \hline
    & 1st \& 2nd elements & 1st \& 3rd elements & 2nd \& 3rd elements \\
    \hline
    1st Eigenvector & 0.00156 & -0.00192 & -0.04112 \\
    2nd Eigenvector & -0.02318 & -0.00403 & 0.01483 \\
    3rd Eigenvector & -0.02529 & -0.04004 & 0.12524 \\
    \hline
    \hline
  \end{tabular}
  \end{table}

Lastly, we did simulation to verify the rate difference of $\langle \hat\bxi_{j}, \be_{j} \rangle$ for $m = 1$ and $m > 1$, revealed in Theorem \ref{thm3.2} (iii). We choose $n = [10\times1.2^l]$ for $l = 0,\dots,9$, $p = [n^3/100]$, where $[\cdot]$ represents rounding. We set $\lambda_j = 1$ for $j \ge 3$ and consider two situations: (1) $\lambda_1 = p, \lambda_2 = 1$, (2) $\lambda_1 = 2\lambda_2 = p$. Under both cases, simulations were carried out  500 times and the corresponding angle of empirical eigenvector and truth was calculated for each simulation. The logarithm of the median absolute error of $\langle \hat\bxi_{1}, \be_{1} \rangle - 1/\sqrt{1+ c_1}$ was plotted against $\log(n)$. Under the two situations, the rate of convergence is $O_P(n^{-3/2})$ and $O_P(n^{-1})$ respectively. Thus the slope of the curves should be $-3/2$ for a single spike and $-1$ for two spikes, which is indeed as the case as shown in Figure~\ref{fig:12spike}.

In short, all the simulation results match well with the theoretical results for the ultra high dimensional regime.

\begin{figure}
        \centering
        \includegraphics[width=\textwidth]{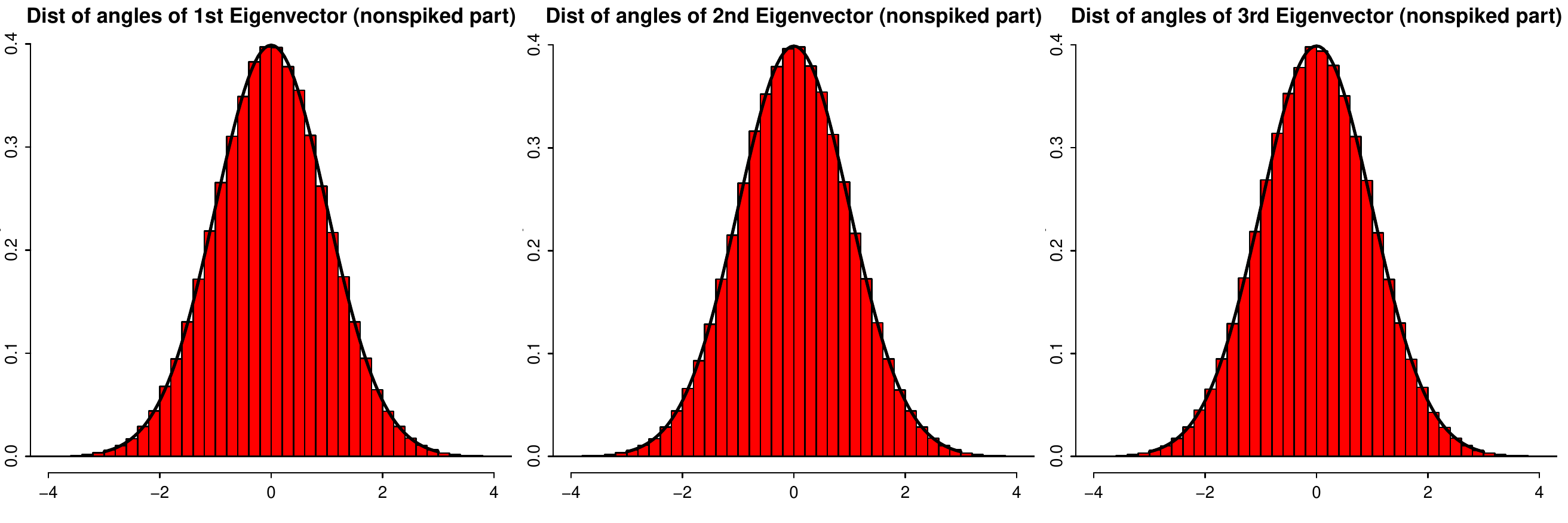}
        \caption{ The empirical distributions of all pairwise angles of the 1000 realized $\hat\bxi_{jB}/\|\hat\bxi_{jB}\|$ $(j=1, 2, 3)$ compared with their asymptotic distributions
        $N(0, 1)$.}\label{fig:eigenAn}
\end{figure}

\begin{figure}
        \centering
       \includegraphics[width=0.4\textwidth]{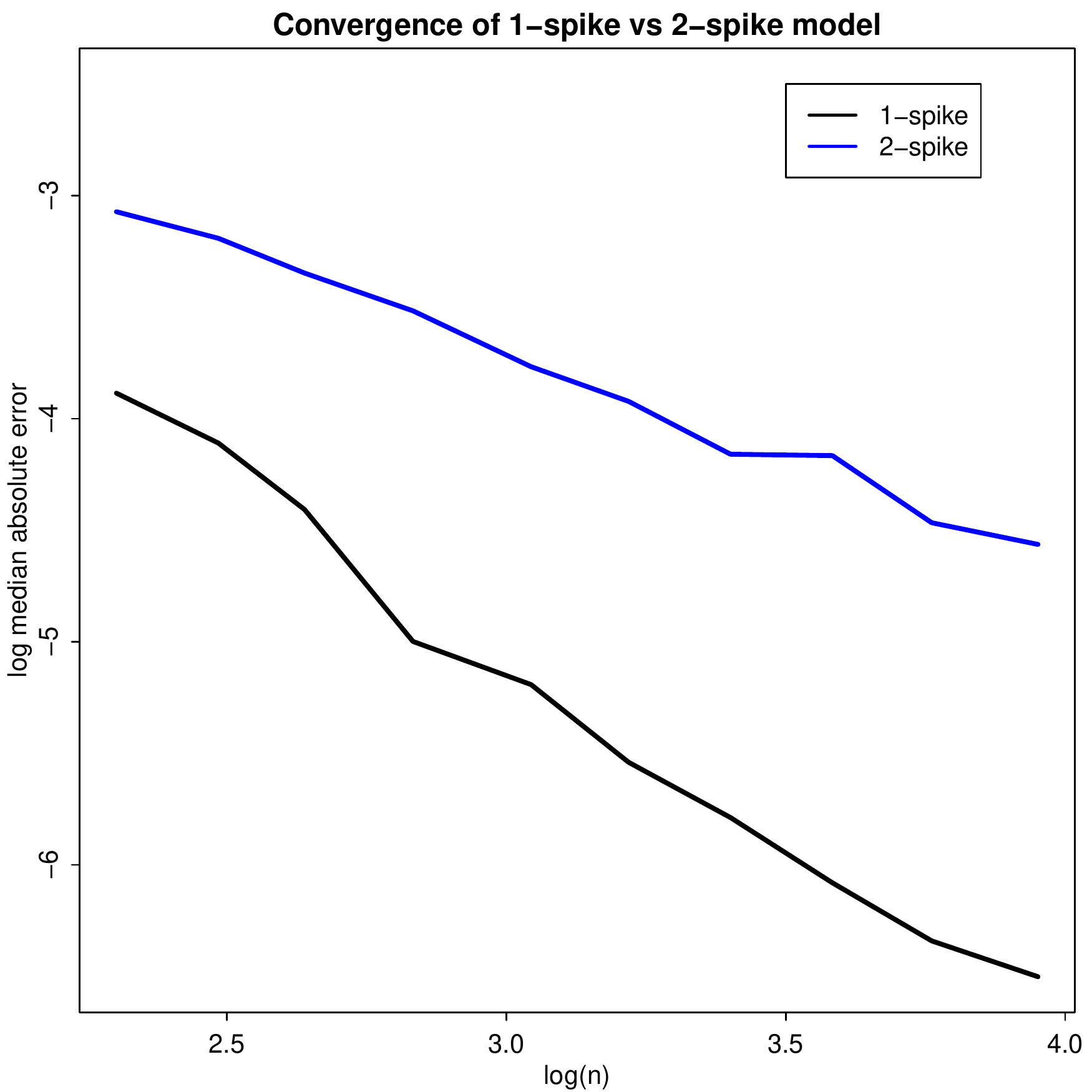}
        \caption{Difference of converged rate of $\langle \hat\bxi_{1}, \be_{1} \rangle - 1/\sqrt{1+ c_1}$ for a single spike model and a two-spike model. The error should be expected to decrease at the rate of $O_P(n^{-3/2})$ and $O_P(n^{-1})$ respectively.}
        \label{fig:12spike}
\end{figure}

\subsection{Performance of S-POET}
We demonstrate the effectiveness of S-POET in comparison with the POET. A similar setting to the last section was used, i.e. $m=3$ and $c_1 = 0.2, c_2 = 0.5, c_3 = 1$. The sample size $T$ ranges from $50$ to $150$ and $p = [T^{3/2}]$. Note that when $T=150$, $p \approx 1800$. The spiked eigenvalues are determined from $p/(T\lambda_j) = c_j$ so that $\lambda_j$ is of order $\sqrt{T}$, which is much smaller than $p$. For each pair of $T$ and $p$, the following steps are used to generate observed data from the factor model for $200$ times.

\begin{itemize}
\item[(1)] Each row of $\bB$ is simulated from the standard multivariate normal distribution and the $j^{th}$ column is normalized to have norm $\lambda_j$ for $j =1, 2, 3$.
\item[(2)] Each row of $\bF$ is simulated from standard multivariate normal distribution.
\item[(3)] Set $\bSigma_u = \diag(\sigma_1^2, \dots, \sigma_p^2)$ where $\sigma_i$'s are generated from Gamma($\alpha,\beta$) with $\alpha = \beta = 100$ (mean $1$, standard deviation $0.1$). The idiosyncratic error $\bU$ is simulated from $N({\bf 0}, \bSigma_u)$.
\item[(4)] Compute the observed data $\bY = \bB\bF' + \bU$.
\end{itemize}

\begin{figure}
        \centering
        \begin{subfigure}[b]{0.47\textwidth}
                \includegraphics[width=\textwidth]{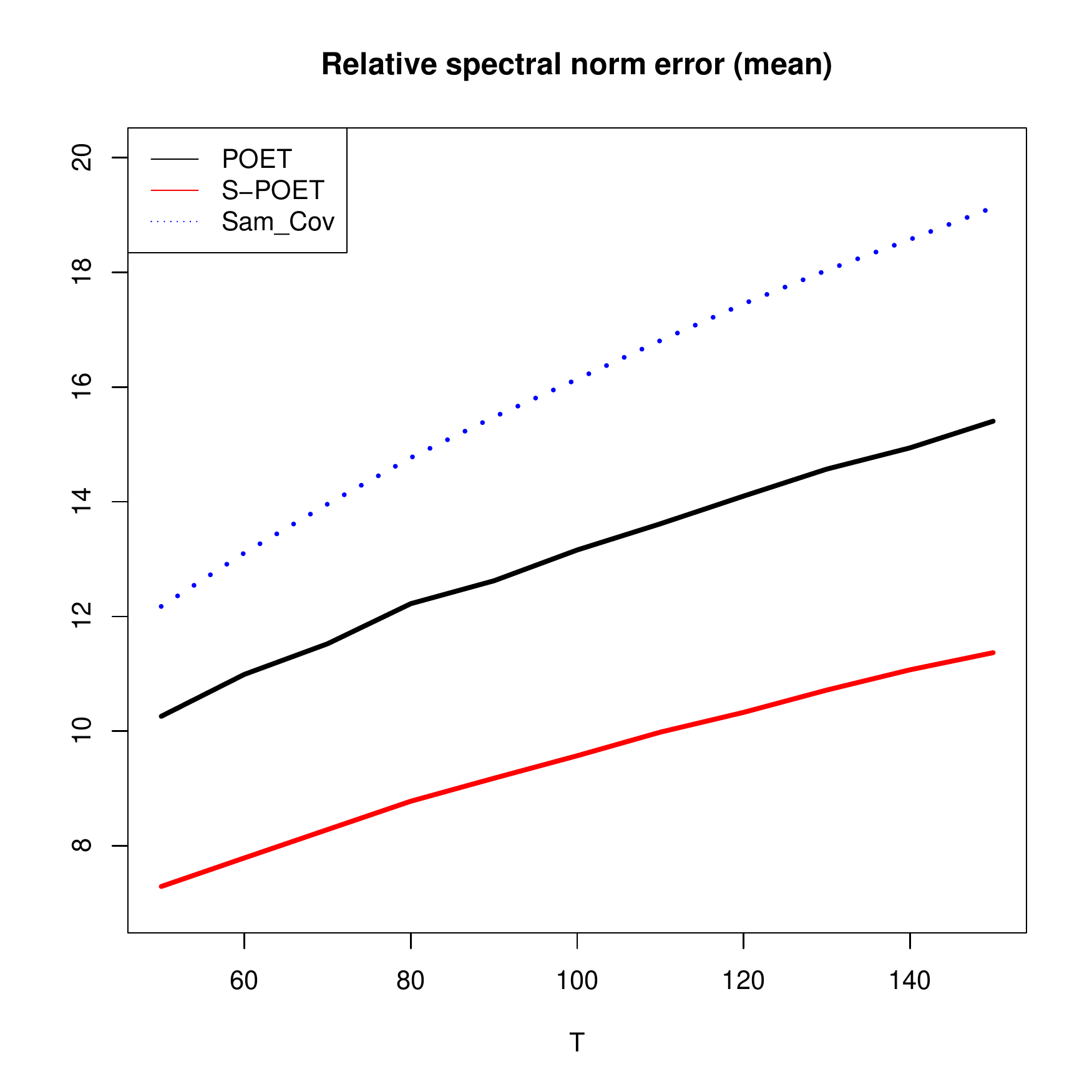}
                \label{fig:normRS}
        \end{subfigure}%
        ~
        \begin{subfigure}[b]{0.47\textwidth}
                \includegraphics[width=\textwidth]{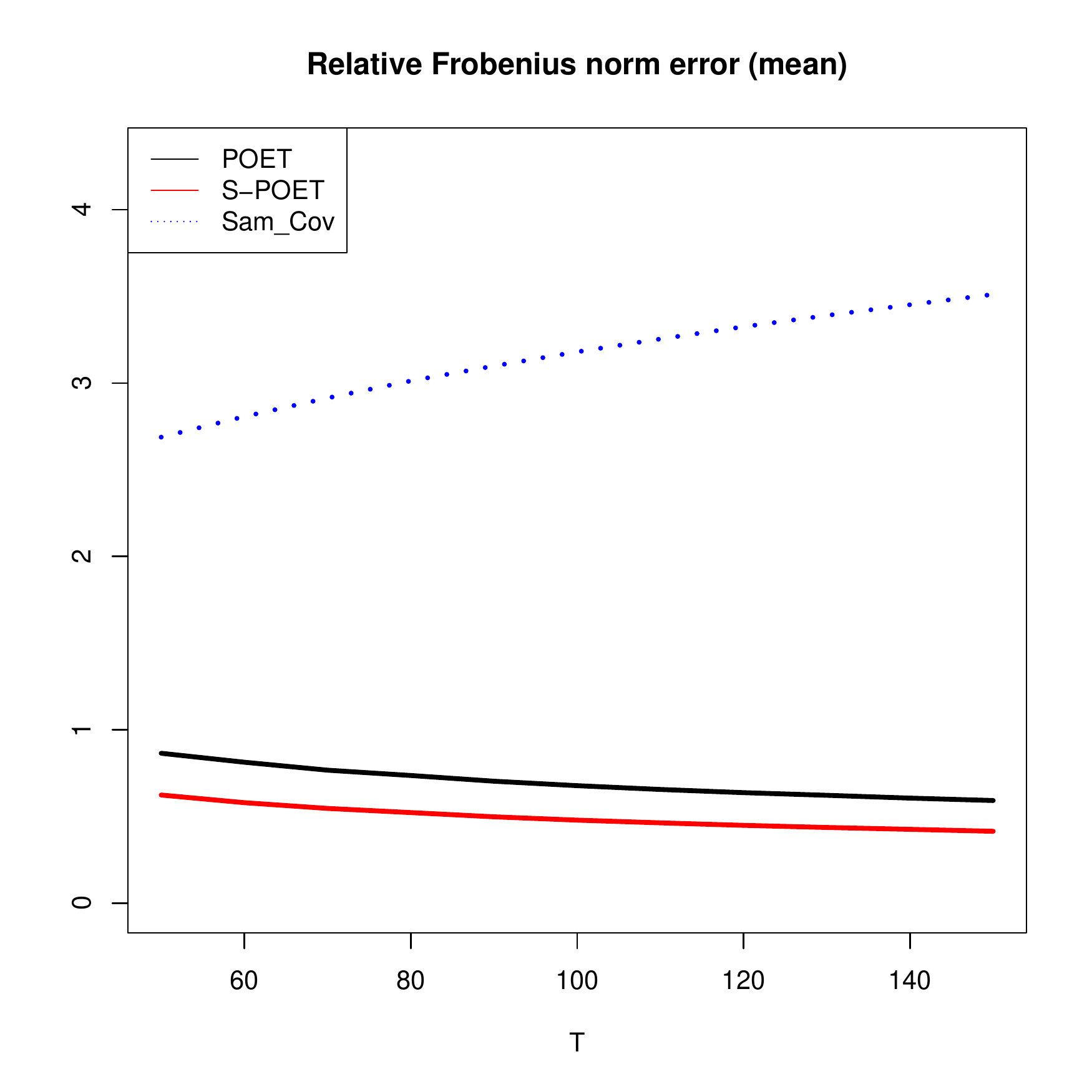}
                \label{fig:normRF}
        \end{subfigure}
        ~
        \begin{subfigure}[b]{0.47\textwidth}
                \includegraphics[width=\textwidth]{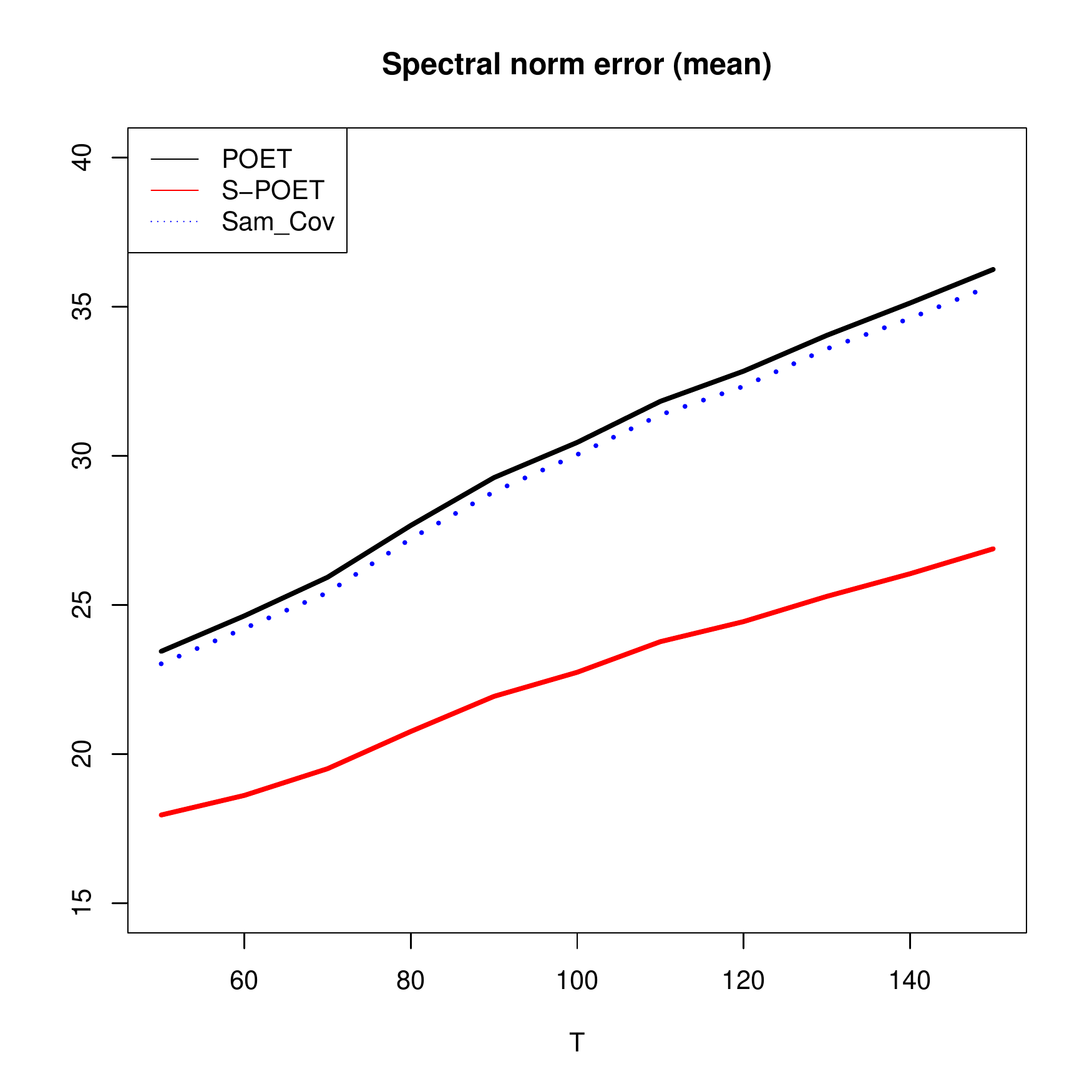}
                \label{fig:normS}
        \end{subfigure}
        ~
        \begin{subfigure}[b]{0.47\textwidth}
                \includegraphics[width=\textwidth]{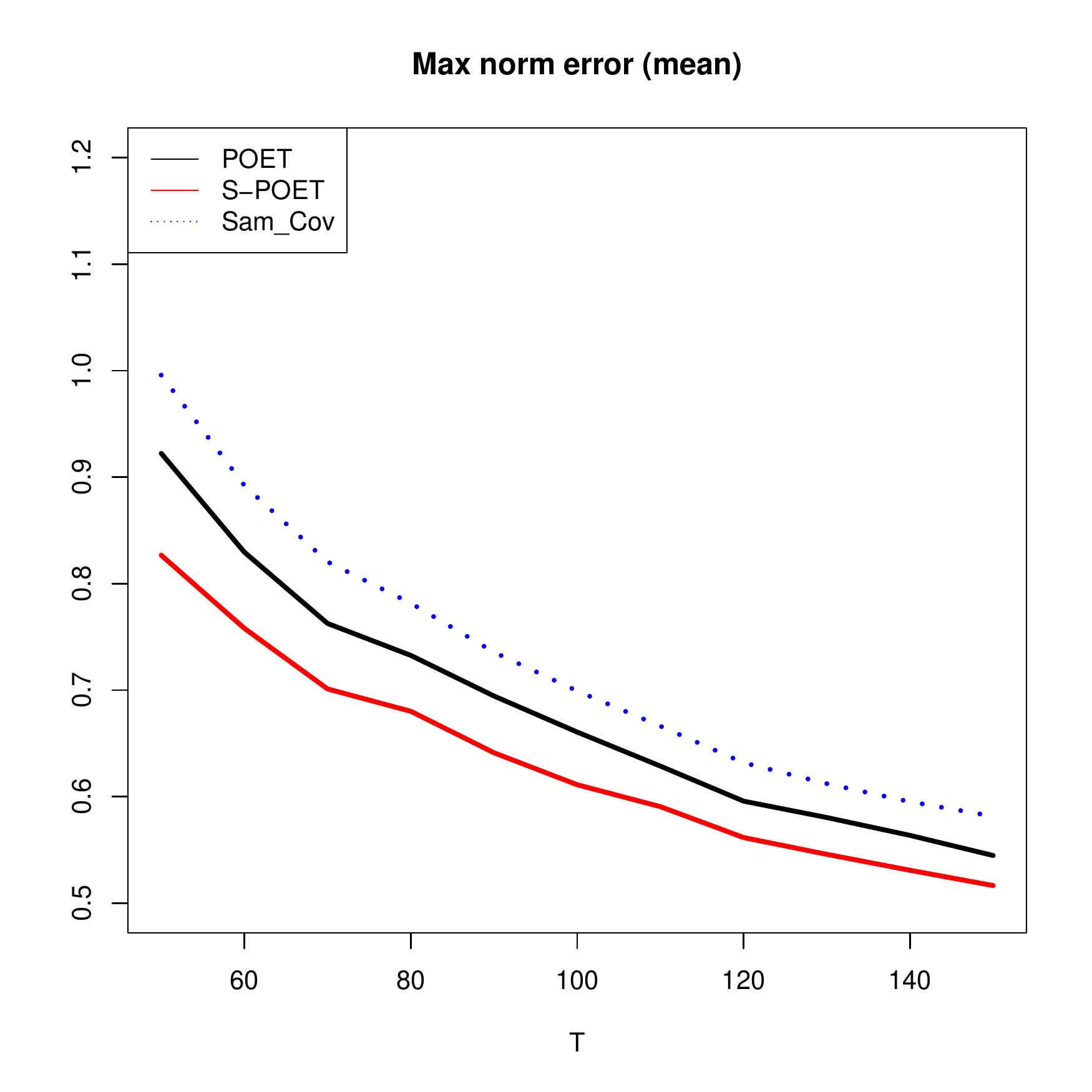}
                \label{fig:normMax}
        \end{subfigure}
        \caption{Estimation error of covariance matrix under respectively relative spectral, relative Frobenius, spectral and max norms using S-POET (red), POET (black) and sample covariance (blue).}\label{fig:errorNorm}
\end{figure}

Both S-POET and POET are applied to estimate the covariance matrix $\bSigma = \bB\bB' + \bSigma_u$. Their mean estimation errors over $200$ simulations, measured in relative spectral norm $\|\hat\bSigma - \bSigma\|_{\bSigma}$, relative Frobenius norm $\|\hat\bSigma - \bSigma\|_{\bSigma, F}$, spectral norm $\|\hat\bSigma - \bSigma\|$ and max norm $\|\hat\bSigma - \bSigma\|_{\max}$, are reported in Figure \ref{fig:errorNorm}. The errors for sample covariance matrix are also depicted for comparison. First notice that no matter in what norm, S-POET uniformly outperforms POET and sample covariance.  It affirms the claim that shrinkage of spiked eigenvalues is necessary to maintain good performance when the spike is not sufficiently large. Since the low rank part is not shrunk for POET, its error under the spectral norm is comparable and even slightly larger than that of the sample covariance matrix. The error under max norm and relative Frobenius norm as expected decreases as $T$ and $p$ increase. However the relative error under the spectral norm does not converge:  our theory shows it should increase in the order $p/T = \sqrt{T}$.

\subsection{FDP estimation} In this section, we report simulation results on FDP estimation by using both POET and S-POET. The data are simulated in a similar way as in Section 5.2 with $p=1000$ and $n=100$. The first $m=3$ eigenvalues have spike size proportional to $p/\sqrt{n}$ which corresponds to $\alpha = \beta = 2/3$ in Theorem \ref{thm_FDP}. The true FDP is calculated by using $\FDP(t) = V(t)/R(t)$ with $t = 0.01$. The approximate FDP, $\FDP_A(t)$, is calculated as in (\ref{FDP_A}) with known $\bB$ but estimated $\bW$ given by $\hat{\bW} = (\bB\bB')^{-1} \bB'\bZ$.  This
$\FDP_A(t)$ based on a known sample covariance matrix serves as a benchmark for our estimated covariance matrix to compare with. We employ POET and S-POET to get $\widehat{\FDP}_{U, POET}(t)$ and $\widehat{\FDP}_{U, SPOET}(t)$.


\begin{figure}
        \centering
        \begin{subfigure}[b]{0.3\textwidth}
                \includegraphics[width=\textwidth]{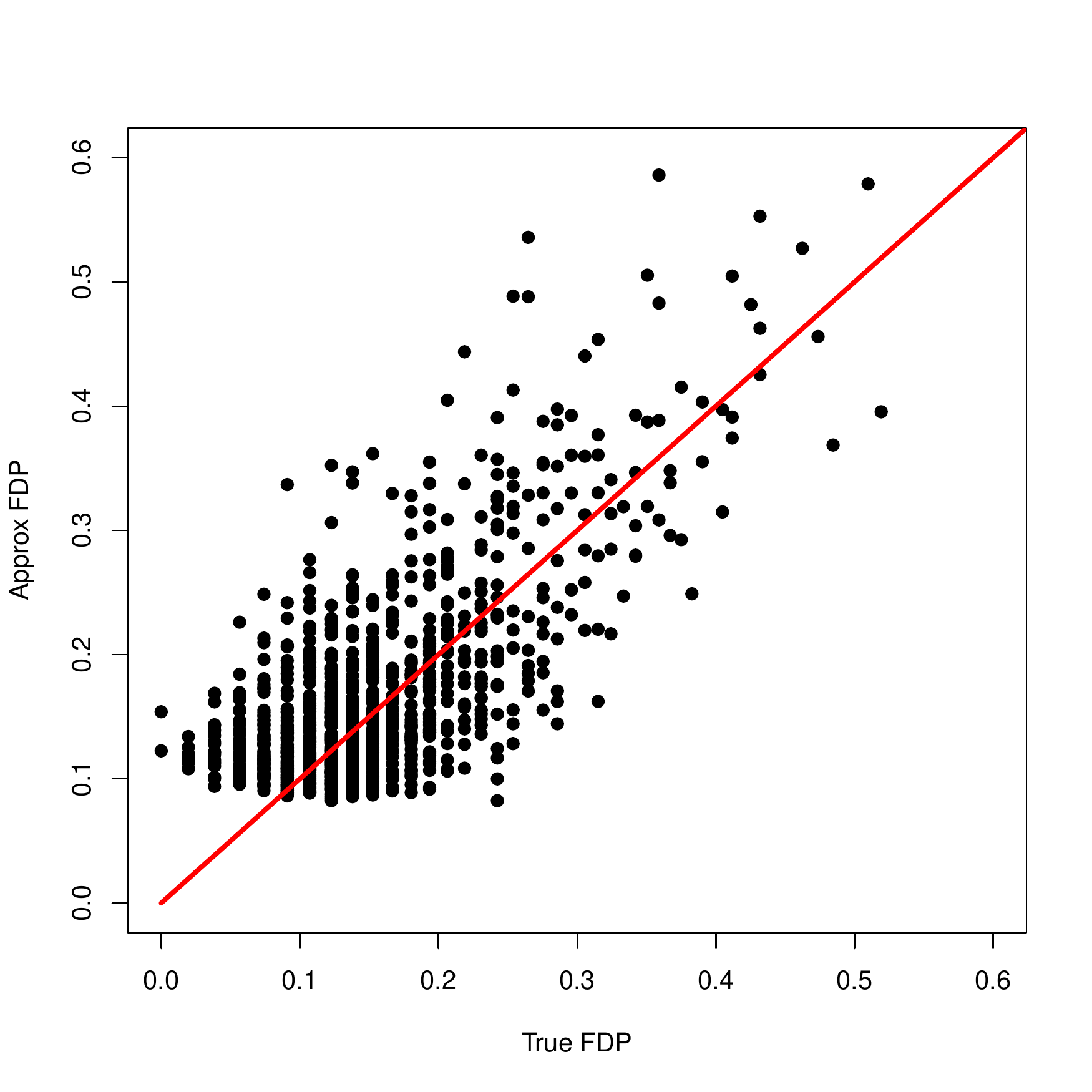}
                \label{fig:approx}
        \end{subfigure}%
        ~
        \begin{subfigure}[b]{0.3\textwidth}
                \includegraphics[width=\textwidth]{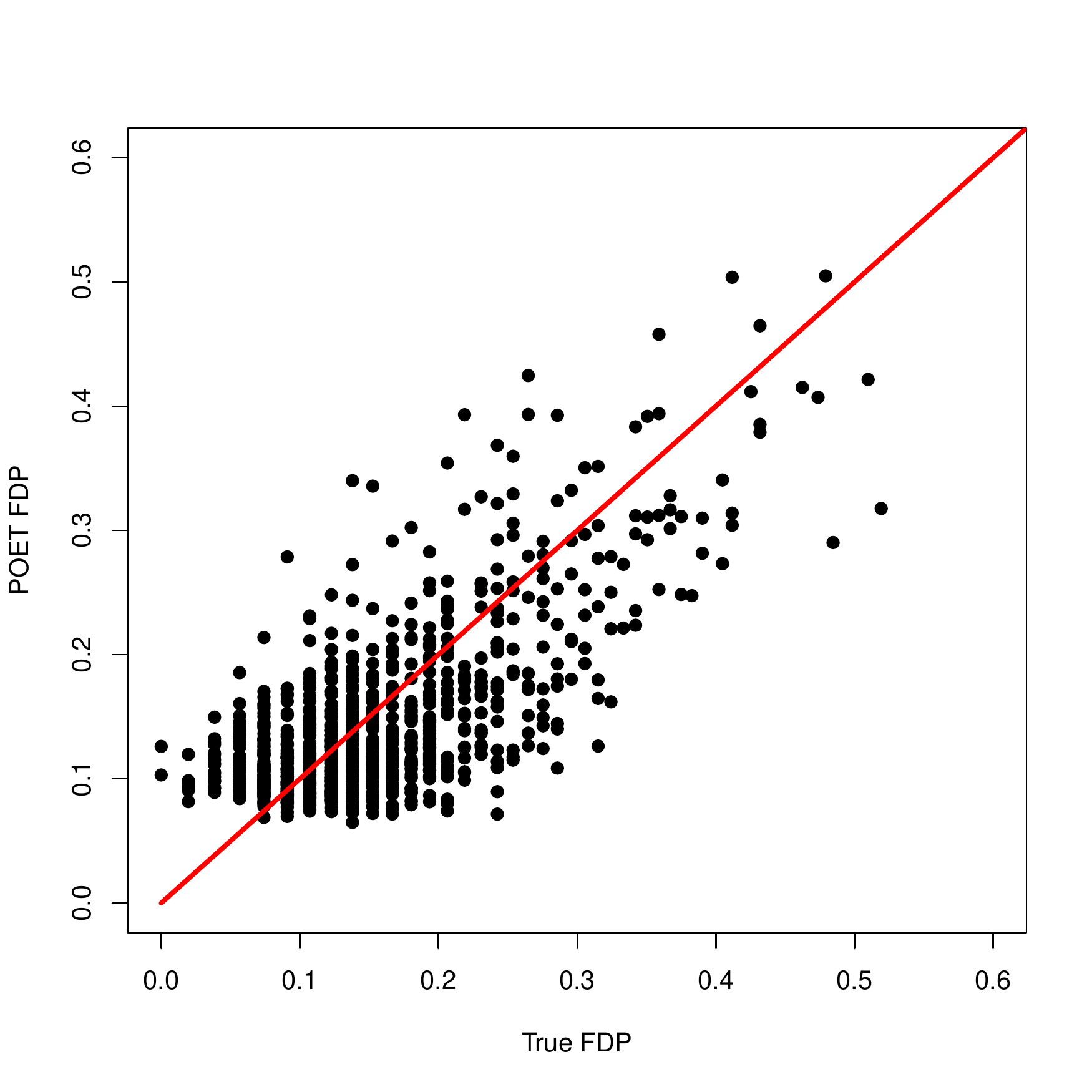}
                \label{fig:POET}
        \end{subfigure}
        ~
        \begin{subfigure}[b]{0.3\textwidth}
                \includegraphics[width=\textwidth]{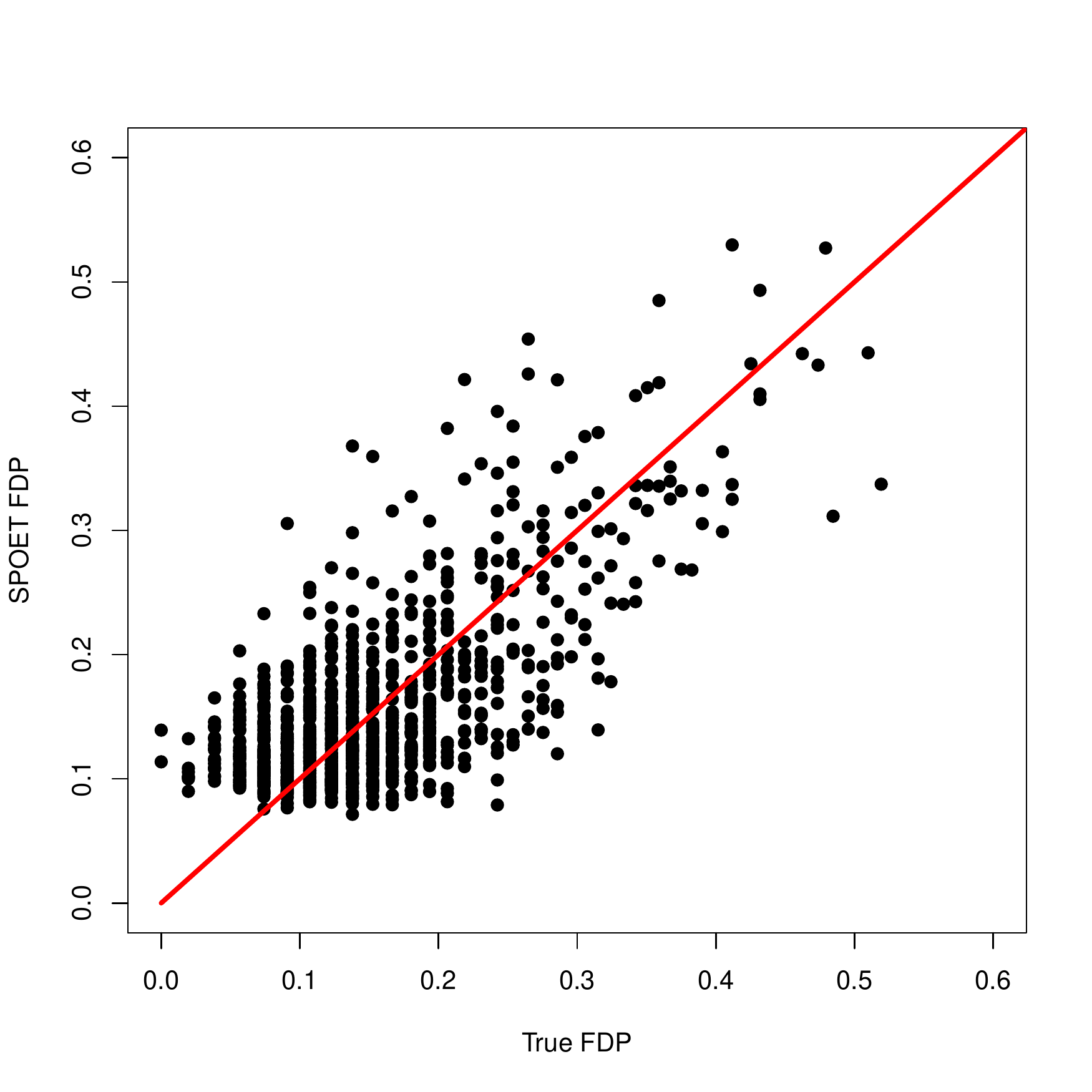}
                \label{fig:SPOET}
        \end{subfigure}
        \caption{Comparison of estimated FDP's with true values. The left plot assumes knowledge of $\bB$, the middle and right ones are corresponding to POET and S-POET methods respectively. The results are aligned along the $45$-degree line, indicating the accuracy of the estimation of FDP. }\label{fig:FDP}
\end{figure}

In Figure \ref{fig:FDP}, three scatter plots are drawn to compare $\FDP_A(t)$, $\widehat{\FDP}_{U, POET}(t)$ and $\widehat{\FDP}_{U, SPOET}(t)$ with the true $\FDP(t)$. The points are basically aligned along the $45$ degree line, meaning that all of them are quite close to the true FDP. With the semi-strong signal $\lambda \propto p/\sqrt{n}$, although much weaker than order $p$, POET accomplishes the task as well as S-POET. 
Both estimators performs as well as if we know the covariance matrix $\bSigma$, the benchmark.  


\section{Proofs for Section~\ref{sec3}} \label{sec6}

\subsection{Proof of Theorem~\ref{thm3.1}}

We first provide three useful lemmas for the proof. Lemma \ref{lem6.1} provides non-asymptotic upper and lower bound for the eigenvalues of weighted Wishart matrix for sub-Gaussian distributions.

\begin{lem} 
\label{lem6.1}
Let $\bA_1, \dots, \bA_n$'s be $n$ independent $p$ dimensional sub-Gaussian random vectors with zero mean and identity variance with the sub-Gaussian norms bounded by a constant $C_0$.  Then for every $t \ge 0$, with probability at least $1-2 \exp(-ct^2)$, one has
\begin{equation*}
\bar w  - \max\{\delta, \delta^2\} \le \lambda_{p}\Big(\frac{1}{n} \sum_{i=1}^n w_i \bA_i \bA_i'\Big)   \le  \lambda_{1}\Big(\frac{1}{n} \sum_{i=1}^n w_i \bA_i \bA_i'\Big)  \le \bar w  + \max\{\delta, \delta^2\} \,.
\end{equation*}
where $\delta = C\sqrt{p/n} + t/\sqrt{n}$ for constants $C,c > 0$, depending on $C_0$. Here $|w_i|$'s is bounded for all $i$ and $\bar w = n^{-1} \sum_{i=1}^n w_i$.
\end{lem}

The above lemma is the extension of the classical Davidson-Szarek bound [Theorem II.7 of \cite{DavSza01}] to the weighted sample covariance with sub-Gaussian distribution. It was shown by \cite{Ver10} that the conclusion holds with $w_i = 1$ for all $i$. With similar techniques to those developed in \cite{Ver10}, we can obtain the above lemma for general bounded weights. The details are omitted.

\medskip
Now in order to prove the theorem, let us define two quantities and treat them separately in the following two lemmas.  Let
\[
\bA = n^{-1} \sum_{j=1}^m \lambda_j \bZ_j \bZ_j', \quad \mbox{and} \quad
\bB = n^{-1} \sum_{j=m+1}^p \lambda_j \bZ_j \bZ_j',
\]
where $\bZ_j$ is columns of $\bX \bLambda^{-\frac12}$.
Then,
\beq
 \label{eq3.2}
   \widetilde \bSigma  = \frac{1}{n} \sum_{j=1}^p \lambda_j \bZ_j \bZ_j' = \bA + \bB.
\eeq

\begin{lem}
\label{lem6.2}
Under Assumptions \ref{assump1} - \ref{assump3}, as $n \to \infty$,
\[
    \sqrt{n} \Big(\lambda_j(\bA)/\lambda_j - 1 \Big) \overset{d} \Rightarrow N(0,\kappa_j-1),
     \quad \mbox{for $j = 1, \dots , m$}.
\]
In addition, they are asymptotically independent.
\end{lem}

\begin{proof}
Note that $\lambda_k^{-1} \bA = n^{-1} \sum_{j=1}^m (\lambda_j/\lambda_k) \bZ_j \bZ_j'$ has the same eigenvalues as matrix $\lambda_k^{-1} \widetilde \bA = n^{-1} \bar \bZ' \bar \bZ$, where $\bar \bZ$ is an $n \times m$ matrix with i.i.d. rows, which are sub-Gaussian distributed with mean ${\bf 0}$ and variance $\bLambda/\lambda_k$. $\lambda_k$ here is only used for normalization. Therefore, we are in the low dimensional situation as Theorem 1 of \cite{And63}. The differences here are two-fold: on one hand we encounter sub-Gaussian distribution; on the other hand the eigenvalues could diverge with different rates of convergence. The result of \cite{And63} can be extended in both directions.

Extension from Gaussian to sub-Gaussian is trivial. The only difference is that the kurtosis of Gaussian is replaced by that of sub-Gaussian distribution. Extension to diverging eigenvalues requires careful scrutiny of Anderson's original proof. Detailedly, following the notations of \cite{And63} and (2.22) therein, we have
\[
H_j := \sqrt{n} (\lambda_j(\widetilde \bA) - \lambda_j) = U_{jj} - n^{-1/2} (M_{jj} - \lambda_j W_{jj}) \,,
\]
where $U_{jj} = \sqrt{n} (\widetilde A_{jj} - \lambda_j)$,  $M_{jj} = \sum_{k \ne j} F_{jk}^2 (\lambda_k + n^{-\frac12} H_k)$, $W_{jj} = \sum_{k \ne j} F_{jk}^2$ and $F_{jk}$ is the $j^{th}$ element of the $k^{th}$ eigenvector of $\widetilde A$ multiplied by $\sqrt{n}$ for $j \ne k$. We claim $M_{jj}/\lambda_{j}$ and $W_{jj}$ are bounded with high probability, so $H_j/\lambda_j$ and $U_{jj}/\lambda_j$ share the same limiting distribution. Therefore the limiting distribution of $\lambda_j(\widetilde \bA)$ for $j = 1,\dots, m$ are independent and
\[
    \sqrt{n} (\lambda_j(\bA) / \lambda_j - 1)
    = \sqrt{n} ( \lambda_j(\widetilde \bA) / \lambda_j - 1)
    \overset{d} \Rightarrow N(0,(\kappa_j-1)\lambda_j^2/ \lambda_j^2).
\]
So the lemma follows.

It remains to show $M_{jj}/\lambda_{j}$ and $W_{jj}$ are $O_P(1)$. Following the cofactor expansion argument of Section 7 in \cite{And63}, it is not hard to see $F_{jk} = O_P(\sqrt{\lambda_j \lambda_k/(\lambda_j - \lambda_k)^2})$. By assumption, $|\lambda_j - \lambda_k| > \delta_0 \max\{\lambda_j, \lambda_k\}$. Hence $F_{jk} = O_P(1/\delta_0) = O_P(1)$ and so is $W_{jj}$. In addition, $M_{jj}/\lambda_j = O_P(\sum_{k \ne j} \lambda_k^2/(\lambda_j - \lambda_k)^2) = O_P(m/\delta_0^2) = O_P(1)$. Now the derivation is complete.
\end{proof}

\begin{lem}
\label{lem6.3}
Under Assumptions \ref{assump1} - \ref{assump3}, for $j = 1, \cdots, m$, we have
\[
    \lambda_k(\bB)/\lambda_j = \bar c c_j + O_P\Big(\lambda_j^{-1} \sqrt{p/n }\Big) + o_P(n^{-\frac 12}), \quad \mbox{for $k =1,2,\dots,n$}.
\]
\end{lem}

\begin{proof}
By definition of $\bB$, $\bB = n^{-1} \bZ_B \bLambda_B \bZ_B'$ where $\bZ_B$ is $n \times (p-m)$ random matrix with independent sub-Gaussian entries of zero mean and unit variance and $\bLambda_B$ is the diagonal matrix with entries $\lambda_{m+1}, \cdots, \lambda_p$. By Lemma \ref{lem6.1} with $t = \sqrt{n}$, for any $k \le n$,
\[
\frac{n}{p-m} \lambda_k(\bB) = \frac{1}{p-m}\sum_{j=m+1}^p \lambda_j + O_P\Big(\sqrt{\frac{n}{p}}\Big) = \bar c + O_P\Big(\sqrt{\frac{n}{p}} \Big) + o_P(n^{-1/2})\,.
\]
Therefore,
\begin{equation*}
\begin{aligned}
\frac{\lambda_k(\bB)}{\lambda_j} & = \frac{n \lambda_k(\bB)}{p-m} \frac{p-m}{n \lambda_j}
& = \bar c c_j + O_P\Big(\lambda_j^{-1} \sqrt{\frac{p}{n}}\Big) + o_P(c_j n^{-\frac 12})\,.
\end{aligned}
\end{equation*}

\end{proof}

\begin{proof}[Proof of Theorem~\ref{thm3.1}]


By Wely's Theorem, $\lambda_j(\bA) + \lambda_n(\bB) \le \hat\lambda_j \le \lambda_j(\bA) + \lambda_1(\bB).$
Therefore from Lemma \ref{lem6.3},
\[
\frac{\hat \lambda_j}{\lambda_j} = \frac{\lambda_j(\bA)}{\lambda_j} + \bar c c_j + O_P\Big(\lambda_j^{-1} \sqrt{\frac{p}{n}}\Big) + o_P(c_j n^{-1/2})\,,
\]
By Lemma \ref{lem6.2} and Slutsky's theorem, we conclude that $\sqrt{n} \Big( \hat \lambda_j/\lambda_j - \Big(1+ \bar c c_j + O_P(\lambda_j^{-1} \sqrt{p/n}) \Big) \Big)$ converges in distribution to $N(0,\kappa_j-1)$ and the limiting distributions of the first $m$ eigenvalues are independent.
\end{proof}

\subsection{Proofs of Theorem~\ref{thm3.2}}
The proof of Theorem~\ref{thm3.2} is mathematically involved. The basic idea for proving part (i) is outlined in Section \ref{sec2}. We relegate less important technical lemmas to the end of the proof in order not to distract the readers. The proof of part (ii) utilizes the invariance of standard Gaussian distribution under orthogonal transformations.

\begin{proof}[Proof of Theorem~\ref{thm3.2}]
(i) Let us start by proving the asymptotic normality of $\hat\bxi_{jA}$ for the case $m > 1$. Write
\[
{\bX} = (\bZ_A \bLambda_A^{\frac{1}{2}}, \bZ_B \bLambda_B^{\frac{1}{2}}) = (\sqrt{\lambda_1} \bZ_1, \dots, \sqrt{\lambda_m} \bZ_m, \sqrt{\lambda_{m+1}}  \bZ_{m+1}, \dots, \sqrt{\lambda_{p}} \bZ_p)\,,
\]
where each $\bZ_j$ follows a sub-Gaussian distribution with mean $\bf 0$ and identity variance $\bI_n$. Then by the eigenvalue relationship of equation (\ref{VecRel}), we have
\beq
\label{eqn4.2}
\hat\bxi_{jA} = \frac{\bLambda_A^{\frac12} \bZ_A' \bu_j}{\sqrt{n \hat\lambda_j}} \;\; \text{and} \;\; \bu_j = \frac{{\bX} \hat\bxi_j}{\sqrt{n \hat\lambda_j}} = \frac{\bZ_A \bLambda_A^{\frac12} \hat\bxi_{jA}}{\sqrt{n\hat\lambda_j}} + \frac{\bZ_B \bLambda_B^{\frac12} \hat\bxi_{jB}}{\sqrt{n \hat \lambda_j}}\,.
\eeq
Recall $\bu_j$ is the eigenvector of the matrix $\tilde\bSigma$, that is, $\frac{1}{n} {\bX}{\bX}' \bu_j = \hat \lambda_j \bu_j$. Using $\bX = (\bZ_A \bLambda_A^{\frac{1}{2}}, \bZ_B \bLambda_B^{\frac{1}{2}})$, we obtain
\begin{equation}
\label{eqn4.3}
\Big( \bI_n - \frac{1}{n} \bZ_A \frac{\bLambda_A}{\lambda_j} \bZ_A' \Big) \bu_j = \bD \bu_j - \Delta \bu_j\,,
\end{equation}
where we denote $\bD = (n\lambda_j)^{-1} \bZ_B\bLambda_B \bZ_B' - \bar c c_j \bI_n$, $\Delta = \hat \lambda_j/\lambda_j - (1+\bar c c_j)$.
We then left-multiply equation (\ref{eqn4.3}) by $\bLambda_A^{\frac12} \bZ_A'/\sqrt{n \hat\lambda_j}$ and employ relationship (\ref{eqn4.2}) to replace $\bu_j$ by $\hat \bxi_{jA}$ and $\hat\bxi_{jB}$ as follows:
\begin{equation}
\label{eqn4.4}
\begin{aligned}
\Big(\bI_m - \frac{\bLambda_A}{\lambda_j} \Big) \hat\bxi_{jA} = & \frac{\bLambda_A^{\frac12} (\frac 1 n \bZ_A' \bZ_A - \bI_m ) \bLambda_A^{\frac12} }{\lambda_j} \hat\bxi_{jA} + \frac{\bLambda_A^{\frac12} \bZ_A' \bD \bZ_A \bLambda_A^{\frac12}}{n \hat\lambda_j} \hat\bxi_{jA} \\ & + \frac{\bLambda_A^{\frac12} \bZ_A' \bD \bZ_B \bLambda_B^{\frac12} }{n \hat\lambda_j} \hat\bxi_{jB} - \Delta \hat\bxi_{jA}\,.
\end{aligned}
\end{equation}
Further define
\[
\bR = \sum_{k \in [m] \setminus j} \frac{\lambda_j}{\lambda_j - \lambda_k} \be_{kA} \be_{kA}'\,.
\]
Then we have $\bR (\bI - \bLambda_A/\lambda_j ) = \bI_m - \be_{jA} \be_{jA}'$. Note that $\bR$ is only well defined if $m > 1$. Therefore, by left multiplying $\bR$ to equation (\ref{eqn4.4}),
\begin{equation}
\label{eqn4.5}
\begin{aligned}
\hat\bxi_{jA}- \langle \hat\bxi_{jA}, \be_{jA} \rangle \be_{jA} = & \bR \Big(\frac{\bLambda_A}{\lambda_j}\Big)^{\frac12} \bK \Big(\frac{\bLambda_A}{\lambda_j}\Big)^{\frac12} \hat\bxi_{jA} \\
& + \bR \frac{\bLambda_A^{\frac12} \bZ_A' \bD \bZ_B \bLambda_B^{\frac12} }{n \hat\lambda_j} \hat\bxi_{jB} - \Delta \bR \hat\bxi_{jA}\,,
\end{aligned}
\end{equation}
where $\bK = n^{-1} \bZ_A' \bZ_A - \bI_n + \lambda_j (n\hat \lambda_j)^{-1} \bZ_A' \bD \bZ_A $. Dividing both side by $\|\hat\bxi_{jA}\|$, we are able to write
\begin{equation}
\label{mainExpression_u}
\frac{\hat\bxi_{jA}}{\|\hat\bxi_{jA}\|} - \be_{jA} = \bR \Big(\frac{\bLambda_A}{\lambda_j}\Big)^{\frac12} \bK \Big(\frac{\bLambda_A}{\lambda_j}\Big)^{\frac12} \be_{jA} + \br_n\,,
\end{equation}
where
\beq
 \label{mainExpression_r}
  \begin{aligned}
\br_n =  & \Big(\langle \frac{\hat\bxi_{jA}}{\|\hat\bxi_{jA}\|}, \be_{jA} \rangle - 1 \Big) \be_{jA} + \bR \Big(\frac{\bLambda_A}{\lambda_j}\Big)^{\frac12} \bK \Big(\frac{\bLambda_A}{\lambda_j}\Big)^{\frac12} \Big( \frac{\hat\bxi_{jA}}{\|\hat\bxi_{jA}\|} - \be_{jA} \Big)  \\
& + \bR \frac{\bLambda_A^{\frac12} \bZ_A' \bD \bZ_B \bLambda_B^{\frac12}}{n \hat\lambda_j} \frac{\hat\bxi_{jB}}{\|\hat\bxi_{jA}\|} - \Delta \bR \Big( \frac{\hat\bxi_{jA}}{\|\hat\bxi_{jA}\|} - \be_{jA} \Big)\,.
  \end{aligned}
\eeq

We will show in Lemma \ref{lem6.4} below that $\br_n$ is a smaller order term.  By Lemma \ref{lem6.4}, noticing that $(\bLambda_A/\lambda_j)^{\frac12} \be_{jA} = \be_{jA}$,
\beq \label{eq6.8}
\sqrt{n} \Big(\frac{\hat\bxi_{jA}}{\|\hat\bxi_{jA}\|} - \be_{jA} + O_P\Big(\sqrt{\frac{p}{n\lambda_j^2}}\Big)  \Big) = \sqrt{n} \bR \Big(\frac{\Lambda_A}{\lambda_j}\Big)^{\frac12} \bK \be_{jA} + o_P(1)\,.
\eeq
Now let us derive normality of the right hand side of (\ref{eq6.8}). According to definition of $\bR$,
\begin{equation}
\label{eqn_R}
\bR \Big(\frac{\bLambda_A}{\lambda_j}\Big)^{\frac12} = \sum \limits_{k \in [m] \setminus j} \frac{\sqrt{\lambda_j \lambda_k}}{\lambda_j - \lambda_k} \be_{kA} \be_{kA}'  \to \sum \limits_{k \in [m] \setminus j} a_{jk} \be_{kA} \be_{kA}'\,.
\end{equation}
Let $\bW = \sqrt{n} \bK \be_{jA} = (W_1,\dots,W_m)$ and $\bW^{(-j)}$ be the $m-1$ dimension vector without the $j^{th}$ element in $\bW$. Since the $j^{th}$ diagonal element of $\bR$ is zero, $\bR (\bLambda_A/\lambda_j)^{\frac12} \bW$ depends only on $\bW^{(-j)}$. Therefore, by Lemma \ref{lem6.5} below and Slutsky's theorem,
\[
\sqrt{n} \bR \Big(\frac{\bLambda_A}{\lambda_j}\Big)^{\frac12} \bK \be_{jA} + O_P\Big(\sqrt{\frac{p}{n\lambda_j^2}}\Big)  \overset{d} \Rightarrow N_m\Big({\bf 0}, \sum\limits_{k \in [m] \setminus j} a_{jk}^2 \be_{kA} \be_{kA}'\Big)\,.
\]
Together with (\ref{eq6.8}), we concludes (\ref{eeq3.3}) for the case $m > 1$.

\medskip
Now let us turn to the case of $m=1$. Since $\bR$ is not defined for $m = 1$, we need to find a different derivation. Equivalently, (\ref{eqn4.3}) can be written as
\[
\frac 1n \bZ_1 \bZ_1' \bu_1 + \frac{1}{n\lambda_1} \bZ_B \bLambda_B \bZ_B' \bu_1 = \frac{\hat\lambda_1}{\lambda_1} \bu_1\,.
\]
Left-multiplying $\bu_1'$ and using relationship (\ref{eqn4.2}), we obtain easily
\[
\hat\xi_{1A}^2 = 1 - \frac{\bar c c_1}{\hat\lambda_1/\lambda_1} - \frac{\lambda_1}{\hat\lambda_1} \bu_1' \bD\bu_1 = 1 - \frac{\bar c c_1}{\hat\lambda_1/\lambda_1} + O_P(\lambda_1^{-1} \sqrt{p/n})\,,
\]
where $\bD$ is defined as before and $\| \bD \| = O_P(\lambda_1^{-1} \sqrt{p/n})$ according to Lemma \ref{lem6.4}. Expanding $\sqrt{1-\bar c c_1/x}$ at the point of $(1+\bar c c_1)$, we have
\[
\hat\xi_{1A} = \frac{1}{\sqrt{1+\bar c c_1}} + \frac{\bar c c_1}{2 (1+\bar c c_1)^{3/2}} \Big(\hat\lambda_1/\lambda_1 - (1+\bar c c_1)\Big) + O_P\Big(\sqrt{\frac{p}{n \lambda_1^2}} + c_1 n^{-1}\Big)\,.
\]
Note that from Lemmas \ref{lem6.2} and \ref{lem6.3}, $\hat\lambda_1/\lambda_1 - (1+\bar c c_1) = (\|\bZ_1\|^2/n - 1) + O_P(\lambda_1^{-1} (p/n)^{1/2}) + o_P(c_j n^{-1/2})$. Therefore due to the fact $\sqrt{n}(\|\bZ_1\|^2/n - 1)$ is asymptotically $N(0, \kappa_1 - 1)$, we conclude
\[
\frac{2(1+\bar c c_1)^{3/2}}{\bar c c_1} \sqrt{n} \Big(\hat\xi_{1A} - \frac{1}{\sqrt{1+\bar c c_1}} + O_P\Big(\sqrt{\frac{p}{n\lambda_1^2}}\Big) \Big) \overset{d} \Rightarrow N(0, \kappa_1 - 1)\,.
\]
This completes the first part of the proof.

\medskip

(ii) We now prove the conclusion for non-spiked part $\hat\bxi_{jB}$. Recall that $\bX_i$ follows $N({\bf 0}, \bLambda)$. Consider $\bX_i^{R} =  \diag(\bI_m, \bD_0) \bX_i$ where as defined in the theorem $\bD_0 = \diag(\sqrt{\bar c/ \lambda_{m+1}}, \dots, \sqrt{\bar c/\lambda_p})$. Here the index $R$ means rescaled data by $\diag(\bI_m, \bD_0)$. After rescaling, we have $\bX_i^{R} \sim N({\bf 0},  \diag(\bLambda_A, \bar c\bI_{p-m}))$. Correspondingly, the $n\times p$ data matrix $\bX^{R} = \bX \diag(\bI_m, \bD_0) = (\bX_A, \bX_B\bD_0)$ where $\bX_A = \bZ_A \bLambda_A^{\frac{1}{2}}$ and $\bX_B = \bZ_B \bLambda_B^{\frac{1}{2}}$ as the notations before. Assume $\hat \bxi_j^{R}$ and $\bu_j^{R}$ are eigenvectors given by $\hat\bSigma^{R}$ and $\widetilde\bSigma^{R}$ of the rescaled data $\bX^{R}$ and $\hat\bxi_j^{R} = ({\hat \bxi_{jA}^{R}}, {\hat\bxi_{jB}^{R}})'$. It has been proved by \cite{Pau07} that $\bh_0 := \hat\bxi_{jB}^{R}/\|\hat\bxi_{jB}^{R}\|$ is distributed uniformly over the unit sphere and is independent of $\|\hat\bxi_{jB}^{R}\|$ due to the orthogonal invariance of the non-spiked part of $\hat\bxi_{jB}^{R}$. Hence it only remains to link $\hat\bxi_{jB}/\|\hat\bxi_{jB}\|$ with $\bh_0$.

Note that $\widetilde\bSigma = n^{-1} \bX \bX'$ and $\tilde \bSigma^{R} = n^{-1} \bX^{R}{{\bX^{R}}'}$, so
\[
\|\widetilde\bSigma - \widetilde \bSigma^{R}\| = \Big\|\frac{1}{n} \bX_B (\bI - \bD_0^2) \bX_B' \Big\| = \Big\|\frac{1}{n} \sum_{j =m+1}^{p} (\lambda_j - \bar c) \bZ_j \bZ_j \Big\|\,,
\]
where the last term is of order $O_P(\sqrt{p/n})$ by Lemma \ref{lem6.1}. Thus by the $\sin \theta$ theorem of \cite{DavKah70}, $\| \bu_j - \bu_j^{R}\| = O_P(\lambda_j^{-1} \sqrt{p/n})$. Next we convert from $\bu_j$ to $\hat\bxi_{jB}$ using the basic relationship (\ref{VecRel}). We have,
\begin{align*}
\Big\| \bD_0 & \frac{\hat\bxi_{jB}}{\|\hat\bxi_{jB}\|} - \frac{\hat\bxi_{jB}^{R}}{\|\hat\bxi_{jB}^{R}\|} \Big\| = \Big\|\frac{\bD_0 \bX_B' \bu_j}{\sqrt{n\hat\lambda_j}\|\hat\bxi_{jB}\|} - \frac{\bD_0{{\bX_B}'} \bu_j^{R}}{\sqrt{n\hat\lambda_j^{R}}\|\hat\bxi_{jB}^{R}\|} \Big\| \\
& \le \Big\| \frac{\bD_0 \bX_B' \bu_j}{\sqrt{n\lambda_j}} \Big\| \left| \sqrt{\frac{\lambda_j}{\hat\lambda_j \|\hat\bxi_{jB}\|^2}} - \sqrt{\frac{\lambda_j}{\hat\lambda_j^R \|\hat\bxi_{jB}^R\|^2}} \right| +
 \Big\|\frac{ \bD_0 \bX_B' }{\sqrt{n\hat\lambda_j^R} \|\hat\bxi_{jB}^R\|} \Big\| \| \bu_j - \bu_j^{R}\| \\
 & =: I + II\,.
\end{align*}
First we claim $II = O_P(\lambda_j^{-1}\sqrt{p/n})$ since $\| \bu_j - \bu_j^{R}\| = O_P(\lambda_j^{-1}\sqrt{p/n})$, $\|\bX_B'/\sqrt{n\lambda_j}\| = O_P(\sqrt{c_j})$, $\lambda_j /\hat\lambda_j^R = O_P(1)$ and $1/\|\hat\bxi_{jB}^R\| = O_P(1/\sqrt{c_j})$ according to Lemma \ref{lem6.6}. Now we show $I = O_P(\sqrt{n/p}) + o_P(n^{-1/2})$. From the proof of Lemma \ref{lem6.6}, we have
\[
\hat\lambda_j \|\hat\bxi_{jB}\|^2/\lambda_j = \bar c c_j + O_P(\lambda_j^{-1} \sqrt{p/n}) + o_P(c_j n^{-1/2}).
\]
Then some elementary calculation gives the rate of $I$. Therefore, $\| \bD_0 \hat\bxi_{jB}/\|\hat\bxi_{jB}\| - \bh_0\| = O_P(\sqrt{n/p}) + o_P(n^{-1/2})$.
The conclusion (\ref{eeq3.5}) follows.

\medskip
To prove the max norm bound (\ref{eeq3.6}) of $\| \hat\bxi_{jB} \|_{\max}$, we first show $\|\bh_0\|_{\max} = O_P(\sqrt{\log p/p})$. Recall that $\bh_0$  is uniformly distributed on unit sphere of dimension $p-m$. This follows easily from its normal representation.
Let $\bG$ to be $p-m$ dimensional multivariate standard normal distributed, then $\bh_0 \overset{d} = \bG/\|\bG\|$.  It then follows
$$
    \|\bh_0\|_{\max} =  \max_{i \le p-m} |G_i| /\|\bG\| = O_P( \sqrt{\log p/p}).
$$

From the derivation above,
\[
\|\hat\bxi_{jB} \|_{\max} \le \left| \sqrt{\frac{\hat\lambda_j^R \|\hat\bxi_{jB}^R\|^2} {\hat\lambda_j \|\hat\bxi_{jB}\|^2}} \right| \|\bD_0^{-1}\| \|\hat\bxi_{jB} \| \Big(II + \|\bh_0\|_{\max}\Big) \,,
\]
which gives $O_P(\sqrt{c_j} (\sqrt{p/(n\lambda_j^2)} + \sqrt{\log p/p})) = O_P(p/(n \lambda_j^{3/2}) + \sqrt{\log p/(n\lambda_j)})$, given the fact that $\|\hat\bxi_{jB}^{R}\| = O_P(\sqrt{c_j})$ by Lemma \ref{lem6.6}. Thus we are done with the second part of the proof.

\medskip
(iii) The proof for the convergence of $\|\hat\bxi_{jA} \|$ and $\|\hat\bxi_{jB}\|$ are given in Lemma \ref{lem6.6}. If $m = 1$, the result for $\|\hat\bxi_{jA} \|$ directly gives (\ref{eeq3.4}) with the same rate. For $m > 1$, from \ref{lem6.6} we have
$$
\|\hat\bxi_{jA}\|^2 = (1+\bar c c_j)^{-1} + O_P\Big(\sqrt{p/(n\lambda_j^2)} + c_j n^{-1/2}\Big) \,.
$$
On the other hand, from Theorem \ref{thm3.2} (i), $\hat\xi_{jk}^2 = O_P(p/(n\lambda_j^2) + 1/n)$ for $k \le m$ and $k \ne j$. So $\hat\xi_{j1}^2 = (1+\bar c c_j)^{-1} + O_P(\sqrt{p/(n\lambda_j^2)} + c_j n^{-1/2} + 1/n)$, which implies (\ref{eeq3.4}).



\end{proof}

\begin{lem}
\label{lem6.4}
As $n \to \infty$, $\|\br_n \| = O_P(\lambda_j^{-1} \sqrt{p/n} + 1/n)$.
\end{lem}

\begin{proof}
Define $\bv_j = \hat\bxi_{jA}/\|\hat\bxi_{jA}\| - \langle \hat\bxi_{jA}/\|\hat\bxi_{jA}\|, e_{jA} \rangle e_{jA}$ and $\alpha_j$, $\beta_j$ and $\gamma_j$ as follows:
\begin{equation*}
\begin{aligned}
&\alpha_j  = \left\|\bR \Big(\frac{\bLambda_A}{\lambda_j}\Big)^{\frac12} \bK \be_{jA} \right\| \,, \\
&\beta_j  = \left\| \bR \Big(\frac{\bLambda_A}{\lambda_j}\Big)^{\frac12} \bK \Big(\frac{\bLambda_A}{\lambda_j}\Big)^{\frac12} \right\|  + \Delta  \left\| \bR \right\| \,, \\
&\gamma_j  = \left\| \bR \frac{\bLambda_A^{\frac12} \bZ_A' \bD \bZ_B \bLambda_B^{\frac12}}{n \hat\lambda_j} \frac{\hat\bxi_{jB}}{\|\hat\bxi_{jA}\|} \right\|\,.
\end{aligned}
\end{equation*}
We claim that $\gamma_j = O_P(\lambda_j^{-1} \sqrt{p/n})$ and $\alpha_j, \beta_j, \|\bv_j\| = O_P(\lambda_j^{-1} \sqrt{p/n}+n^{-\frac12})$. Then the rate of $\|\br_n\|$ could be easily derived from its definition (\ref{mainExpression_r}) and the above results. To be specific, first notice the following two inequalities: by (\ref{eqn4.5}), $\|\bv_j\| \le \beta_j + \gamma_j$; by orthogonal decomposition $\hat\bxi_{jA}/\|\hat\bxi_{jA}\| = \langle \hat\bxi_{jA}/\|\hat\bxi_{jA}\|, \be_{jA} \rangle \be_{jA} + \bv_j$, we have $1- \langle \hat\bxi_{jA}/\|\hat\bxi_{jA}\|, \be_{jA} \rangle = 1-\sqrt{1-\|\bv_j\|^2} \le \|\bv_j\|^2$. Note that we always choose $\hat \bxi$ so that $\langle \hat\bxi_{jA}/\|\hat\bxi_{jA}\|, \be_{jA} \rangle$ is positive. Therefore
\begin{equation}
\label{eqn4.5_rate1}
\begin{aligned}
\left\|\frac{\hat\bxi_{jA}}{\|\hat\bxi_{jA}\|} - \be_{jA} \right\| & = \left\| \bv_j + \Big(\langle \frac{\hat\bxi_{jA}}{\|\hat\bxi_{jA}\|}, \be_{jA} \rangle -1\Big) \be_{jA} \right\| \\
& \le \|\bv_j\|(1+\|\bv_j\|) \le \|\bv_j\| (1+\beta_j + \gamma_j)\,.
\end{aligned}
\end{equation}
Hence, by (\ref{mainExpression_r}),
\begin{equation*}
\begin{aligned}
\|\br_n\| & \le \|\bv_j\|^2 + \beta_j \left\|\frac{\hat\bxi_{jA}}{\|\hat\bxi_{jA}\|} - \be_{jA} \right\| + \gamma_j \\
& \le \|\bv_j\|(\|\bv_j\| + \beta_j(1+\beta_j + \gamma_j)) + \gamma_j = O_P(\lambda_j^{-1} \sqrt{p/n}+1/n)\,.
\end{aligned}
\end{equation*}

It remains to show the claims above. Let us first show the rate of convergence of $\gamma_j$. In order to prove this, we need the rate of $\|\bD\|$. By Lemma \ref{lem6.1}, $\| (p-m)^{-1} \bZ_B \bLambda_B \bZ_B' - \bar c \bI \| = O_P(\sqrt{n/p})$, so we have
\begin{equation*}
\begin{aligned}
\| \bD \| & = \| \frac{1}{n} \frac{\bZ_B \bLambda_B \bZ_B'}{\lambda_j} - \bar c c_j \bI_n \| \\
& \le \frac{p-m}{n \lambda_j} \| \frac{1}{p-m} \bZ_B \bLambda_B \bZ_B' - \bar c \bI \| + \bar c | \frac{m}{n \lambda_j} | = O_P(\lambda_j^{-1} \sqrt{p/n})\,.
\end{aligned}
\end{equation*}
Hence
$$
\gamma_j \le \left\| \frac{\bR \bLambda_A^{\frac12}}{\sqrt{ \lambda_j}}  \right\| \left\| \frac{\lambda_j}{\hat\lambda_j} \right\|  \left\| \frac{ \bZ_A'}{\sqrt{n}} \right\| \| \bD \| \left\| \frac{\bZ_B\bLambda_B^{\frac{1}{2}}}{{\sqrt{n \lambda_j}}} \right\| \frac{\|\hat\bxi_{jB}\|}{\|\hat\bxi_{jA}\|} = O_P(\lambda_j^{-1} \sqrt{p/n})\,,
$$
since the other terms except $\|\bD\|$ are all $O_P(1)$. Indeed, (\ref{eqn_R}) says the first term is asymptotically bounded. We have shown in the proofs of Lemmas {\ref{lem6.2}} and \ref{lem6.3} that the second, third and fifth terms are $O_P(1)$. In addition, the facts that $\|\hat\bxi_{jB}\| \le 1$, $\|\hat\bxi_{jA}\| \overset{P} \to (1+\bar c c_j)^{-\frac12}$ imply the last term is $O_P(1)$.

Then let us show that $\alpha_j$ and $\beta_j$ are $O_P(\lambda_j^{-1} \sqrt{p/n}+n^{-\frac12})$. The rate of $\|\bK\|$ is needed. By Lemma \ref{lem6.1}, $\| \frac{1}{n} \bZ_A' \bZ_A - \bI \| = O_P(\sqrt{m/n}) = O_P(n^{-\frac12})$.  Thus,
\begin{align*}
\|\bK\| &= \left\| \frac 1 n \bZ_A' \bZ_A - \bI_n + \frac{\lambda_j}{\hat \lambda_j} \frac 1 n \bZ_A' \bD \bZ_A \right\|  \\
& \le \left\| \frac 1 n \bZ_A' \bZ_A - \bI_n \right\| + |\frac{\lambda_j}{\hat \lambda_j}| \| \bD \| \left\| \frac 1 n \bZ_A' \bZ_A \right\| = O_P\Big(\sqrt{\frac {p}{n\lambda_j^2}}+n^{-\frac12}\Big)\,.
\end{align*}
Then easily we get $\alpha_j = O_P(\lambda_j^{-1} \sqrt{p/n}+n^{-\frac12})$. Note that from Theorem \ref{thm3.1} that $\Delta = \hat \lambda_j/\lambda_j - (1+\bar c c_j) = O_P(\lambda_j^{-1} \sqrt{p/n}+n^{-\frac12})$, so
$$
\beta_j \le \| \bK\| \|\bR \bLambda_A / \lambda_j \| + \Delta \|\bR\| = O_P(\lambda_j^{-1} \sqrt{p/n}+n^{-\frac12}) \,,
$$
where similar to (\ref{eqn_R}), $\|\bR \bLambda_A / \lambda_j \|$ and $\|\bR\|$ are $O_P(1)$.

Finally, $\|\bv_j\| \le \beta_j + \gamma_j = O_P(\lambda_j^{-1} \sqrt{p/n}+n^{-\frac12})$.
The proof is complete.
\end{proof}

\begin{lem}
\label{lem6.5}
$\bW^{(-j)} + O_P(\lambda_j^{-1} \sqrt{p/n})\overset{d} \Rightarrow N(\bf 0,\bI_{m-1})$.
\end{lem}

\begin{proof}
Recall $\bW = \sqrt{n} \bK \be_{jA}$. Then, by the definition of $\bK$,
\begin{equation*}
\bW =   \frac{1}{\sqrt{n}} \bZ_A' \bZ_j - \sqrt{n}\be_{jA} + \frac{\lambda_j}{\hat \lambda_j} \frac{1}{\sqrt{n}} \bZ_A' \bD \bZ_j\,.
\end{equation*}
Its $t^{th}$ component is $W_t = n^{-1/2} \bZ_t' \bZ_j + \delta_{nt}$ for $t \in [m] \setminus j$ where $\delta_{nt} = (\lambda_j / \hat \lambda_j) \cdot n^{-1/2} \bZ_t' \bD \bZ_j$. Denote $\tilde \bW = (n^{-1/2} \bZ_t' \bZ_j)_{t \in [m] \setminus j}$ and $\bdelta_n = (\delta_{nt})_{t \in [m] \setminus j}$. We claim as $n \to \infty$, $\|\bdelta_n \| = O_P(\lambda_j^{-1} \sqrt{p/n})$. So $\bW^{(-j)} =  \tilde \bW + O_P(\lambda_j^{-1} \sqrt{p/n})$. In order to prove the lemma, it suffices to show that $\tilde \bW$ follows $N({\bf 0},\bI_{m-1})$. That is, for any vector $\ba$ of $m-1$ dimension, $\mathbb E [ \exp(i \ba' \tilde \bW) ] \to \exp(-\| \ba \|^2/2)$ almost surely.
$$
\mathbb E \Big[ e^{i{\ba}' \tilde \bW} \Big] = \mathbb E \Big[ \mathbb E \Big[ \prod \limits_{t \in [m] \setminus j} e^{i a_t \bZ_t' \bZ_j /\sqrt{n}} | \bZ_j \Big] \Big] = \mathbb E \Big[ \prod_{t \in [m] \setminus j} \prod_{k = 1}^n f_t\Big(\frac{1}{\sqrt{n}} a_t Z_{kj} \Big) \Big],
$$
where $f_j(u) = \mathbb E [ \exp(i u Z_{kj}) ]$ is the characteristic function of each element of $\bZ_j$. The sub index $j$ means we actually allow different characteristic functions for the columns of $\bZ_A$ and $\bZ_B$.

By Taylor expansion, we can easily derive
$$
| e^{ix} - 1 - ix + x^2/2 | \le (|x|^{3}/6) \wedge x^2\,,
$$
from which it holds that
$$
| f_j(u) - 1 - iu\mathbb E[Z_{kj}] + \frac{u^2}{2} \mathbb E[Z_{kj}^2] | \le u^2 \mathbb E\Big[ \frac{|u|}{6} |Z_{kj}|^3 \wedge Z_{kj}^2 \Big]\,.
$$
$\mathbb E\Big[ \frac{|u|}{6} |Z_{kj}|^3 \wedge Z_{kj}^2 \Big]$ goes to $0$ as $u \to 0$ and is dominated by the integrable function $Z_{kj}^2$. So by Dominated Convergence Theorem the right hand side is $o(u^2)$. Therefore, $f_j(u) = 1-u^2/2 + o(u^2)$. Using this result, we have
\begin{align*}
\mathbb E \Big[ e^{i{\ba}' \tilde \bW} \Big] & = \mathbb E \Big[ \prod_{t \in [m] \setminus j} \prod_{k = 1}^n \left( 1 - \frac{a_t^2}{2n} Z_{kj}^2 \right) \Big] + o(1)\\
& = \mathbb E \Big[ \prod_{k = 1}^n \left( 1 - \frac{\|\ba\|^2}{2n} Z_{kj}^2 \right) \Big] + o(1) \\
& = \prod_{k = 1}^n \mathbb E \Big[  1 - \frac{\|\ba\|^2}{2n} Z_{kj}^2 \Big] + o(1) \\
& = \left( 1 - \frac{\|{\bf a}\|^2}{2n} \right)^{n} + o(1) \overset{a.s.} \to \exp(-\|{\bf a}\|^2/2)\,.
\end{align*}
which implies $\tilde W$ follows $N(0,I_{m-1})$.

Now let us validate $\|\bdelta_n \| = O_P(\lambda_j^{-1} \sqrt{p/n})$. Clearly
$$
|\delta_{nt}| \le | \lambda_j/\hat \lambda_j | \left| \frac{1}{\sqrt{n}} \bZ_t' \bZ_j \right| \|\bD\| \,.
$$
We have shown that $| \lambda_j/\hat \lambda_j | = O_P(1)$ and $\|\bD\| = O_P(\lambda_j^{-1} \sqrt{p/n})$. It suffices to show $\frac{1}{\sqrt{n}} \bZ_t' \bZ_j = O_P(1)$.
$$
\mathbb E \Big[  \Big| \frac{1}{\sqrt{n}} \bZ_t'
\bZ_j \Big |^2 \Big] = \frac{1}{n} \mathbb E \Big[ \mathbb E \Big[ (\bZ_t'
\bZ_j) ^2 | \bZ_j \Big] \Big] = \frac{1}{n} \mathbb E \Big[ \bZ_j' \bZ_j \Big] = 1\,.
$$
So by Markov inequality, we have $| n^{-1/2} \bZ_t' \bZ_j|$ is $O_P(1)$, which generates $\delta_{nt} = O_P(\lambda_j^{-1} \sqrt{p/n})$. So is $\|\bdelta_n\|$ since $\bdelta_n$ is of fixed length $m-1$. The proof is complete.
\end{proof}

\begin{lem}
\label{lem6.6}
$\|\hat\bxi_{jA}\| = (1+ \bar c c_j)^{-1/2} + O_P(\lambda_j^{-1}\sqrt{p/n} + c_j n^{-1/2})$ and \\
$\|\hat\bxi_{jB}\| = (\frac{\bar c c_j}{1+ \bar c c_j})^{1/2} + O_P(\sqrt{1/\lambda_j} + \sqrt{c_j} n^{-1/2})$.
\end{lem}

\begin{proof}
If $m = 1$, Theorem \ref{thm3.2} (i) directly implies the conclusions. So in the following, we only consider $m > 1$.
Recall that $\bX = (\bZ_A \bLambda_A^{\frac 1 2}, \bZ_B \bLambda_B^{\frac 1 2})$. Let $\bZ = (\bZ_A, \bZ_B)$, then
$$
\bZ = \bX \bLambda^{-\frac 1 2} = \sqrt{n} \hat\bLambda^{\frac 12} (\hat\bxi_1, \dots, \hat\bxi_n)' \bLambda^{-\frac 1 2}\,,
$$
where $\bLambda = \diag(\bLambda_A, \bLambda_B)$ and $\hat\bLambda = \diag(\hat\lambda_1, \dots, \hat\lambda_n)$.  Define
$$
\bar \bLambda = \diag(1,\dots, 1, \lambda_{m+1},\dots,\lambda_{p})
$$
and consider the eigenvalue of the matrix $n^{-1} \bZ \bar\bLambda \bZ'$. The $j$-th diagonal element of the matrix must lie in between its minimum and maximum eigenvalues. That is
$$
\lambda_n(\frac 1 n \bZ \bar\bLambda \bZ') \le \Big(\frac 1n \bZ \bar\bLambda \bZ'\Big)_{jj} = \hat\lambda_j \sum_{k=1}^p \hat\xi_{jk}^2 \frac{\bar\lambda_k}{\lambda_k} \le \lambda_1(\frac 1 n \bZ \bar\bLambda \bZ')\,,
$$
where $\hat\xi_{jk}$ is the $k$-th element of the $j$-th empirical eigenvector for $j \le m$. Divided by $\hat\lambda_j$, then by Theorem \ref{thm3.1} and Lemma \ref{lem6.1} both the left and right hand side converge to
$$
\frac{\bar c c_j}{1+\bar c c_j} + O_P\Big(\sqrt{\frac{p}{n\lambda_j^2}} + c_j n^{-1/2}\Big) \,.
$$
So $\sum_{k=1}^p \hat\xi_{jk}^2 \bar\lambda_k/\lambda_k$ also converges to the above quantity. Also, by definition, $\bar\lambda_k/\lambda_k = O_P(\lambda_k)$ for $k \le m$ while the ratio is 1 for $k > m$. By Theorem \ref{thm3.2} (i), $\hat\xi_{jk}^2 = O_P(n^{-1}\lambda_j\lambda_k/(\lambda_j - \lambda_k)^2)$ for $j \ne k \le m$. Hence, $\|\hat\bxi_{jB}\|^2 = \sum_{k=m+1}^p \hat\xi_{jk}^2$ again converges to the above quantity, which implies the rates of convergence for $\|\hat\bxi_{jB}\|$ and $\|\hat\bxi_{jA}\| = (1-\|\hat\bxi_{jB}\|^2)^{1/2}$.
\end{proof}

\appendix

\section{Comparison on assumptions} \label{secA}
The following assumptions are from \cite{FanLiaMin13}, where the results were established for the mixing sequence. But we only consider i.i.d. data in this paper. The  assumptions are listed for completeness and comparison with Assumptions \ref{assump1:factor} and \ref{assump2:factor}.

\begin{assum} \label{assPoet1}
$\| p^{-1} \bB'\bB -\bOmega_0\| = o(1)$ for some $m \times m$ symmetric positive definite matrix $\bOmega_0$ such that $\bOmega_0$ has $m$ distinct eigenvalues and that $\lambda_{\min}(\bOmega_0)$ and $\lambda_{\max}(\bOmega_0)$ are bounded away from both zero and infinity.
\end{assum}

\begin{assum} \label{assPoet2}
(i) $\{\bu_t, \bff_t\}_{t \ge 1}$ is strictly stationary. In addition, $\mathbb E[u_{it}] = \mathbb E[u_{it} f_{jt}] = 0$ for all $i \le p, j \le m$ and $t \le T$. \\
(ii) There exist positive constants $c_1$ and $c_2$ such that $\lambda_{\min}(\bSigma_u) > c_1$, $\|\bSigma_u\|_{\infty} < c_2$, and $\min_{i,j} \Var(u_{it} u_{jt}) > c_1$. \\
(iii) There exist positive constants $r_1, r_2, b_1$ and $b_2$ such that for $s>0, i\le p, j \le m$,
\[
\mathbb P(|u_{it}| > s) \le exp(-(s/b_1)^{r_1}) \; \text{ and } \; \mathbb P(|f_{jt}| > s) \le exp(-(s/b_2)^{r_2})\,.
\]
\end{assum}

We introduce the strong mixing conditions. Let $\mathcal F_{-\infty}^0$ and $\mathcal F_{n}^{\infty}$ denote the $\sigma$-algebras generated by $\{(\bff_s, \bu_s): -\infty \le s \le 0\}$ and $\{(\bff_s, \bu_s): n \le s \le \infty\}$ respectively. In addition, define the mixing coefficient
\[
\alpha(n) = \sum_{A \in \mathcal F_{-\infty}^0, B \in \mathcal F_{n}^{\infty}} |\mathbb P(A) \mathbb P(B) - \mathbb P(AB)| \,.
\]

\begin{assum} \label{assPoet3}
There exists $r_3 > 0$ such that $3 r_1^{-1} + 1.5 r_2^{-1} + r_3^{-1} > 1$ and $C>0$ satisfying $\alpha(n) \le \exp(-Cn^{r_3})$ for all $n$.
\end{assum}
Note that for the independence case, Assumption \ref{assPoet3} is trivially satisfied since $\alpha(n) = 0$ for all $n$.

\begin{assum} \label{assPoet4}
There exists $M > 0$ such that for all $i \le p$ and $s, t \le T$, \\
(i) $\|\bb_i\|_{\max} \le M$, \\
(ii) $\mathbb E[p^{-1/2} (\bu_s'\bu_t - \mathbb E \bu_s'\bu_t)]^4 \le M$, \\
(iii) $\mathbb E\|p^{-1/2} \sum_{i=1}^p \bb_i u_{it}\|^4 \le M$.
\end{assum}

\section{Proofs of Theorems in Section 4} \label{secB} 
In order to prove theorems in Section \ref{sec4}, convergence rate of the sparse error matrix $\bSigma_u$ is required. The following theorem states the convergence rate for estimating $\bSigma_u$ by the thresholding procedure in (\ref{thresholding}). Its proof and related technical lemmas are given in Appendix \ref{secC}. 

\begin{thm} \label{thm_thresholding}
Under the assumptions of Theorem \ref{thm_RelSpectral}, by applying adaptive thresholding estimator (\ref{thresholding}) with $\tau_{ij} = C \omega_T (\hat\sigma_{u,ii} \hat\sigma_{u,jj})^{1/2}$ and $\omega_T = \sqrt{\log p/T} + \sqrt{1/p}$, we have
$$
\|\hat\bSigma_u^{\top} - \bSigma_u\| = O_P(\omega_T^{1-q} m_p)
$$
\end{thm}

Given Theorem~\ref{thm_thresholding}, we are ready to start showing theorems in Section \ref{sec4}. The proofs are built based on conclusions in Section \ref{sec3}.

\begin{proof}[{Proof of Theorem~\ref{thm_RelSpectral}}]

We first prove the theorem for term $\Delta_{L1}$. Write $\bB = (\widetilde\bb_1, \dots, \widetilde\bb_m)$ and the minimizer of (\ref{minimization}) as $\hat\bB = (\hat\bb_1, \dots, \hat\bb_m)$. Since $\hat\bB$ is just the eigenvectors (unnormalized) of $\hat\bSigma$, we have:
\[
\hat\lambda_j = \|\hat\bb_j\|^2 \text{ and } \hat\xi_j = \hat\bb_j/\|\hat\bb_j\| \,.
\]
Then $\hat\lambda_j^S = \|\hat\bb_j\|^2 - \bar c p/n$ or $\hat\lambda_j^S = \|\hat\bb_j\|^2 - \hat c p/n$ if $\bar c$ is unknown. Let $\hat \bLambda = \diag(\|\hat\bb_1\|^2,\dots,\|\hat\bb_m\|^2)$ be the diagonal matrix of the first $m$ empirical eigenvalues and $\hat\bGamma = (\hat\bb_1/\|\hat\bb_1\|, \dots, \hat\bb_m/\|\hat\bb_m\|)$ be the empirical eigenvector matrix. In Sections \ref{sec3.1} and \ref{sec3.2}, our results for empirical eigenvalues and eigenvectors imply the following:
\beq \label{Eqn6.9}
\|\bLambda^{-1/2} (\hat\bLambda^S - \bLambda) \bLambda^{-1/2}\| = O_P(\lambda_m^{-1} \sqrt{p/T}+T^{-1/2})\,;
\eeq
and
\beq \label{Eqn6.10}
\|\hat\bGamma'\bGamma - \bD \| = O_P(\lambda_m^{-1} \sqrt{p/T}+T^{-1/2})\,,
\eeq
where $\bD = \diag((1+\bar c c_1)^{-1/2}, \dots, (1+\bar c c_m)^{-1/2})$. Now let us start to bound $\Delta_{L1}$ and $\Delta_{L2}$.
\begin{align*}
\Delta_{L1} & \le \|\bLambda^{-1/2} \bGamma' (\hat\bGamma \hat\bLambda^S \hat\bGamma' - \bGamma\bLambda\bGamma') \bGamma \bLambda^{-1/2}\| +  \|\bLambda^{-1/2} \bGamma' (\bGamma\bLambda\bGamma' - \bB\bB') \bGamma \bLambda^{-1/2}\| \\
& = : \Delta_{L1}^{(1)} + \Delta_{L1}^{(2)}\,.
\end{align*}
We handle the two terms separately.
\[
\Delta_{L1}^{(1)} \le \|\bLambda^{-1/2} \bGamma' (\hat\bGamma \hat\bLambda^S \hat\bGamma' - \bGamma\bD\bLambda\bD\bGamma') \bGamma \bLambda^{-1/2}\| \,,
\]
where we used $\bD^2 \preceq \bI$. The right hand side is further bounded by $I + 2 II + III$ with
\begin{align*}
I & = \|\bLambda^{-1/2} (\bGamma' \hat\bGamma - \bD) \hat\bLambda^S (\hat\bGamma' \bGamma - \bD) \bLambda^{-1/2}\| \,, \\
II & = \|\bLambda^{-1/2} (\bGamma' \hat\bGamma - \bD) \hat\bLambda^S \bD \bLambda^{-1/2}\|, \;\; III = \|\bLambda^{-1/2} \bD (\hat\bLambda^S - \bLambda) \bD \bLambda^{-1/2}\| \,.
\end{align*}
By equations (\ref{Eqn6.9}) and (\ref{Eqn6.10}), we conclude that $II$ and $III$ are of order $O_P(\lambda_m^{-1} \sqrt{p/T} + T^{-1/2})$ and $I$ is of smaller order. Thus $\Delta_{L1}^{(1)} = O_P(T^{-1/2})$. In order to derive rate of $\Delta_{L1}^{(2)}$, denote $\tilde\bLambda = \diag(\|\tilde\bb_1\|^2,\dots,\|\tilde\bb_m\|^2)$ and $\tilde\bGamma = (\tilde\bb_1/\|\tilde\bb_1\|, \dots, \tilde\bb_m/\|\tilde\bb_m\|)$ so that $\bB\bB' = \tilde\bGamma \tilde\bLambda \tilde\bGamma'$. We could treat $\Delta_{L1}^{(2)}$ similar to $\Delta_{L1}^{(1)}$. $\Delta_{L1}^{(2)}$
could be bounded by $I' + 2 II' + III'$ with
\begin{align*}
I' & = \|\bLambda^{-\frac12} (\bGamma' \tilde\bGamma - \bI) \tilde\bLambda (\tilde\bGamma' \bGamma - \bI) \bLambda^{-\frac12}\|, \\
II' & = \|\bLambda^{-\frac12} (\bGamma' \tilde\bGamma - \bI) \tilde\bLambda \bLambda^{-\frac12}\|,\; III' = \|\bLambda^{-\frac12} (\tilde\bLambda - \bLambda) \bLambda^{-\frac12}\| \,.
\end{align*}
By Weyl's theorem, $| \lambda_j - \|\tilde\bb_j\|^2| \le \|\bSigma_u\| \le C$, so $III' = O(1/\lambda_m)$. By $sin \theta$ theorem,
\[
\|\bGamma' \tilde\bGamma - \bI\| = \|\bGamma' (\tilde\bGamma - \bGamma)\| \le \|\tilde\bGamma - \bGamma\| \le C\|\bSigma_u\|/\lambda_m = O(1/\lambda_m)\,,
\]
so is $II'$. Since $I'$ is of smaller order, we conclude $\Delta_{L1}^{(2)} = O(1/\lambda_m)$. Therefore, $\Delta_{L1} \le \Delta_{L1}^{(1)} + \Delta_{L1}^{(2)} = O_P(T^{-1/2})$.

\medskip
The bound for term $\Delta_{L2}$ is derived in the following. Recall that
\[
\Delta_{L2} = \|\bTheta^{-\frac12} \bOmega' (\hat\bGamma \hat\bLambda^S \hat\bGamma' - \bB \bB') \bOmega \bTheta^{-\frac12} \|\,,
\]
which is bounded by
\[
\|\bTheta^{-\frac12} \bOmega' \hat\bGamma \hat\bLambda^S \hat\bGamma' \bOmega \bTheta^{-\frac12} \| + \|\bTheta^{-\frac12} \bOmega' \tilde\bGamma \tilde\bLambda \tilde\bGamma' \bOmega \bTheta^{-\frac12} \| =: \Delta_{L2}^{(1)} + \Delta_{L2}^{(2)}\,.
\]
\[
\Delta_{L2}^{(1)} \le \|\bTheta^{-1}\| \|\bOmega'\hat\bGamma\|^2 \|\hat\bLambda^S\| = O_P(p/T)\,,
\]
because by Lemma \ref{lem6.6}, $\|\bOmega'\hat\bGamma\| = O_P(\sqrt{c_m}) = O_P(\sqrt{p/(T\lambda_m)})$.
\[
\Delta_{L2}^{(2)} \le \|\bTheta^{-1}\| \|\bOmega'\tilde\bGamma\|^2 \|\tilde\bLambda\| = O_P(1/\lambda_m)\,,
\]
as $\|\bOmega'\tilde\bGamma\| = \|\bGamma\bGamma' - \tilde\bGamma\tilde\bGamma'\| = O(\|\bSigma_u\|/\lambda_m) = O_P(1/\lambda_m)$ by $\sin \theta$ Theorem. Finally, $\Delta_{L2}^{(1)} = O_P(p/T + 1/\lambda_m)$.

Finally let us look at term $\Delta_S$. Since $\Delta_S \le \|\bSigma^{-1}\| \|\hat\bSigma_u^{\top} - \bSigma_u\|$, it suffices to bound $\|\hat\bSigma_u^{\top} - \bSigma_u\|$, which has already been done in Theorem \ref{thm_thresholding}. So
\[
\Delta_S = O_P\Big(m_p \Big(\frac{\log p}{T}+\frac1p\Big)^{(1-q)/2}\Big) \,.
\]

\end{proof}

\begin{proof}[{Proof of Theorem \ref{thm_Risk}}]

The numerator of the relative risk is bounded by
\[
|\bw'(\hat\bGamma \hat\bLambda^{S} \hat\bGamma' - \bB\bB') \bw| + |\bw' (\hat\bSigma_u^{\top} - \bSigma_u) \bw| \,.
\]
The second term is bounded by $\|\hat\bSigma_u^{\top} - \bSigma_u\|\|\bw\|^2$, thus is $O_P(\Delta_S\|\bw\|^2)$. By using $\bw = (\bGamma, \bOmega) \bfeta$, the first term can be written as
\[
|(\bfeta_A' \bGamma' + \bfeta_B'\bOmega')\bE(\bGamma\bfeta_A + \bOmega\bfeta_B)| = O_P(\bfeta_A' \bGamma' \bE \bGamma\bfeta_A + \bfeta_B'\bOmega'\bE \bOmega\bfeta_B)\,,
\]
where $\bE = \hat\bGamma \hat\bLambda^{S} \hat\bGamma' - \bB\bB'$. It is easy to see from the proof of Theorem~\ref{thm_RelSpectral} that
\[
\bfeta_A' \bGamma' \bE \bGamma\bfeta_A = O_P(\Delta_{L1}\lambda_1\|\bfeta_A\|^2) \,,
\]

By Theorem~\ref{thm3.2}, $\|\bOmega' \hat\bGamma\|_{\max} = \max_{j} \|\hat\bxi_{jB}\|_{\max} = O_P(p/(T\lambda_m^{3/2}) + \sqrt{\log p/(n\lambda_m)})$. From proof of Theorem~\ref{thm_RelSpectral},
we know that $\| \bOmega' \tilde\bGamma\|_{\max} \le \| \bOmega' \tilde\bGamma\| = O_P(1/\lambda_m)$. Therefore,
\[
\bfeta_B'\bOmega'\bE \bOmega\bfeta_B \le \|\bfeta_B\|_1^2 O_P( \|\bOmega' \hat\bGamma\|_{\max}^2 \| \hat\bLambda^S\| + \| \bOmega' \tilde\bGamma\|_{\max}^2 \|\tilde\bLambda \|)\,,
\]
which gives $\bfeta_B'\bOmega'\bE \bOmega\bfeta_B = O_P(c_m^2 + \log p/T + \lambda_m^{-1})$.

The denominator is lower bounded by $\bw'\bSigma\bw \ge \lambda_m \|\bfeta_A\|^2 + c\|\bfeta_B\|^2$.
Thus the relative risk is of order
\[
O_P\Big(\frac{\Delta_{L1}\lambda_1\|\bfeta_A\|^2 + c_m^2 + \log p/T + \lambda_m^{-1}}{\lambda_m \|\bfeta_A\|^2 + c\|\bfeta_B\|^2} +\Delta_S \Big) = O_P\Big(T^{-\min \{\frac{2(\alpha+\beta-1)}{\beta}, \frac12\}} + m_p w_T^{1-q}\Big)\,,
\]
for $\alpha < 1$ if we plug in the convergence rate of $\Delta_{L1}$ and $\Delta_S$ in Theorem \ref{thm_RelSpectral}. If $\alpha \ge 1$, the relative risk is $O_P(T^{-1/2} + m_p w_T^{1-q})$.
Note the rate $O_P(T^{-2(\alpha+\beta-1)/\beta})$ comes from $c_m^2$ in the numerator. If we further assume $\|\bfeta_A\| \ge C_2$, this rate becomes $c_m^2/\lambda_m$ dominated by $T^{-1/2}$, thus the relative risk is again of order $O_P(T^{-1/2} + m_p w_T^{1-q})$.
\end{proof}

\medskip
\begin{proof}[{Proof of Theorem \ref{thm_FDP}}]

The proof follows Theorem 1 of \cite{FanHan13}. Using their notation, we have
\[
\widehat{\FDP}_U(t) - \FDP_A(t) = (\Delta_1 + \Delta_2)/R(t) + O(p^{\theta-1/2}\|\bmu^*\|)\,,
\]
where with $\widetilde\bW = (\bB'\bB)^{-1} (\bB'\bZ)$,
\begin{eqnarray*}
\Delta_{1} & = &\sum_{i = 1}^p \Big[ \Phi(\hat a_i (z_{t/2} + \hat\bb_i' \hat\bW)) - \Phi(a_i(z_{t/2} +\bb_i' \widetilde\bW))\Big] \,,\\
\Delta_{2} & = & \sum_{i = 1}^p \Big[ \Phi(\hat a_i (z_{t/2} - \hat\bb_i' \hat\bW)) - \Phi(a_i(z_{t/2} -\bb_i' \widetilde\bW))\Big].
\end{eqnarray*}
We just need to bound $\Delta_1$, then $\Delta_2$ can be bound similarly. As shown in \cite{FanHan13},
\[
|\Delta_1| \le C\Big(\sum_{j=1}^m |\hat\lambda_j^S - \lambda_j| + \lambda_j\|\hat\bxi_j^S - \bxi_j\| +  \sqrt{p}(\|\bmu^*\|+\sqrt{p}) \|\hat\bxi_j^S - \bxi_j\|  \Big)\,,
\]
where $\hat\lambda_j^S$ and $\hat\bxi_j^S$ are the $j^{th}$ eigenvalue and eigenvector of $\hat\bSigma^S$ defined in (\ref{Eqn:SPOET}). So by Weyl's theorem and Theorem \ref{thm_RelSpectral},
\begin{align*}
|\hat\lambda_j^S - \lambda_j| & \le \|\hat\bSigma^S - \bSigma\| = O_P(\Delta_{L1}\lambda_1 + \Delta_{L2} + \Delta_S) \\
& = O_P\Big(\frac{\lambda_1 }{\sqrt{T}} + \sqrt{\frac{\log p}{T}}+\frac{p}{T}\Big) = O_P\Big(\frac{p}{T} + \frac{p^{\alpha}}{\sqrt{T}}\Big)\,.
\end{align*}
By $sin \theta$ theorem, we also have $\|\hat\bxi_j^S - \bxi_j\| \le O_P(\|\hat\bSigma^S - \bSigma\|/\lambda_j)$. So finally
$$
|\Delta_1/R(t)| = O_P\Big( p^{\theta} \Big(\frac{1}{T} + \frac{p^{\alpha-1}}{\sqrt{T}} + \Big(\frac{\|\bmu^*\|}{\sqrt{p}} + 1\Big) \Big(\frac{p^{1-\alpha}}{T} + \frac{1}{\sqrt{T}}\Big) \Big) \Big)\,.
$$
Since $Cp^{1-\alpha} < Cp^{\beta} \le T$, 
$$
|\widehat{\FDP}_U(t) - \FDP_A(t)| = O_P\Big(p^{\theta}(\|\bmu^*\|p^{-\frac12} + T^{-\min\{\frac{\alpha+\beta-1}{\beta}, \frac12\}})\Big)\,.
$$

\end{proof}

\section{Convergence rate of error matrix} \label{secC}
In order to achieve convergence rate Theorem \ref{thm_thresholding} for the covariance matrix of idiosyncratic error, we employ the following lemma from \cite{FanLiaMin13}.

\begin{lem}  \label{lem_Basic}
Suppose that $(\log p)^{6\alpha} = o(T)$ where $\alpha = 3r_1^{-1}+1$ and Assumption \ref{assump1:factor} and \ref{assump2:factor} hold. In addition, suppose that there is a sequence $a_T = o(1)$ so that $\max_{i \le p} T^{-1} \sum_{t=1}^T |\hat u_{it} - u_{it}|^2 = O_P(a_T^2)$ and $\max_{i\le p, t\le T} |\hat u_{it} - u_{it}| = o_P(1)$. Then there is a constant $C>0$ in the adaptive thresholding estimator (\ref{thresholding}) with $\tau_{ij} = C \omega_T (\hat\sigma_{u,ii} \hat\sigma_{u,jj})^{1/2}$ and
\[
\omega_T = \sqrt{\frac{\log p}{T}} + a_T\,,
\]
such that
\[
\|\hat\bSigma_u^{\top} - \bSigma_u\| = O_P(\omega_T^{1-q} m_p)\,.
\]
\end{lem}

The essential step of applying the previous lemma is to find $a_T$. We start by getting the convergence rate of $\hat\bF$ and $\hat\bB$. Let $\bV$ denote the $m \times m$ diagonal matrix of the first $m$ largest eigenvalues of the sample covariance matrix in decreasing order. Recall that
\[
\frac{1}{T} \bY' \bY \hat\bF = \hat\bF \bV\,.
\]
Define
\[
\bH = \frac{1}{T} \bV^{-1} \hat\bF' \bF \bB' \bB\,.
\]
\begin{lem} \label{lc.2}
The rates of convergence of $\hat\bF$ are as follows: \\
(i) $\| \hat\bF - \bF \bH'\|_F = O_P(\frac{p}{\lambda_m \sqrt{T}} + \sqrt{\frac{T}{\lambda_m}})$, \\
(ii) $\| \hat\bF - \bF \bH'\|_{\max} = O_P((\frac{1}{\sqrt{\lambda_m}} + \frac{p}{\lambda_m T} + \frac{\sqrt{p}}{\lambda_m} ) (\log T)^{\frac{2}{r_2}} )$,\\
\end{lem}
\begin{proof} (i) By definition of $\hat\bF$ and $\bH$
\[
\hat\bF - \bF \bH' = \frac{1}{T} (\bY'\bY - \bF\bB' \bB \bF') \hat\bF \bV^{-1}\,.
\]
Since $\|\hat\bF\|_F = O_P(\sqrt{T})$, $\|\bV^{-1}\| = O_P(1/\lambda_m)$ from Theorem \ref{thm3.1}, we have
\[
\| \hat\bF - \bF \bH' \|_F \le O_P\Big(\frac{1}{\lambda_m \sqrt{T}}\Big) \| \bU'\bU + \bF \bB' \bU + \bU' \bB \bF'\|\,,
\]
where we used the fact $\|\bA\bB\|_F \le \|\bA\| \|\bB\|_F$. By Lemma \ref{lem6.1},
\[
\|\frac{1}{T}\bU'\bU\| = \|\frac{1}{T}\bU \bU'\| \le \|\frac{1}{T}\bU\bU' - \bSigma_u\| + \|\bSigma_u\| = O_P(\frac{p}{T})\,,
\]
and since $\|\bB\|_{\max} = O_P(\sqrt{\lambda_1/p})$ from Assumption \ref{assump2:factor},
\[
\mathbb E\|\bB' \bU \|_F^2 = \sum_{t = 1}^T \sum_{j=1}^m \mathbb E(\sum_{i=1}^p u_{it} b_{ij})^2 \le \sum_{t = 1}^T \sum_{j=1}^m \sum_{i_1=1}^p \sum_{i_2=1}^p |\sigma_{u,i_1 i_2}| O(\frac{\lambda_1}{p}) = O(T\lambda_1)\,.
\]
Therefore by Markov inequalty,
$$
\|\bF \bB' \bU \| \le \|\bF\|_F \|\bB' \bU\| = O_P(T\sqrt{\lambda_1})\,.
$$
Hence,
$$
\| \hat\bF - \bF \bH' \|_F = O_P\Big(\frac{p}{\lambda_m \sqrt{T}} + \sqrt{\frac{T}{\lambda_m}}\Big)\,.
$$
(ii) From (i) we conclude
$$
\| \hat\bF - \bF \bH' \|_{\max} \le O_P\Big(\frac{1}{\lambda_m T}\Big) \| \bU'\bU\hat\bF + \bF \bB' \bU\hat\bF + \bU' \bB \bF' \hat\bF\|_{\max}\,.
$$
Let us bound each term separately. For the first term, $\| \bU'\bU\hat\bF\|_{\max} \le \| \bU'\bU\|_{\infty} \|\hat\bF\|_{\max}$ and
$$
\| \bU'\bU\|_{\infty} = \max_{t} \sum_{s = 1}^T |\bu_t' \bu_s| = \max_{t} \sum_{s = 1}^T \Big|\bu_t' \bu_s - \mathbb E[\bu_t' \bu_s]\Big| + \mathbb E[\bu_t' \bu_t] = O_P(T\sqrt{p} + p)\,.
$$
The second term is bounded as $\|\bF \bB' \bU\hat\bF\|_{\max} \le m \|\bF\|_{\max} \| \bB' \bU\|_{\infty} \|\hat\bF\|_{\max}$ and
$$
 \| \bB' \bU\|_{\infty} = \max_{k \le m} \sum_{t=1}^T |\tilde\bb_k' \bu_t| = O(T\sqrt{\lambda_1})\,,
$$
since $\var(\tilde\bb_k' \bu_t) = \tilde\bb_k' \bSigma_u \tilde\bb_k = O(\lambda_1)$. The third term can be bounded similarly. Together with the fact that $\|\bB\|_{\max} = O_P(\sqrt{\lambda_1/p})$ and $\|\bF\|_{\max} = O_P((\log T)^{1/r_2})$ from Assumption \ref{assump2:factor}, we obtain
\begin{align*}
\| \hat\bF - \bF \bH' \|_{\max} & \le O_P\Big(\frac{1}{\lambda_m T}\Big) \Big((p+T\sqrt{p})(\log T)^{\frac{1}{r_2}}+ T\sqrt{\lambda_1} (\log T)^{\frac{2}{r_2}} \Big) \\
& = O_P\Big( \Big(\frac{1}{\sqrt{\lambda_m}} + \frac{p}{\lambda_m T} + \frac{\sqrt{p}}{\lambda_m} \Big) (\log T)^{\frac{2}{r_2}} \Big)\,.
\end{align*}
\end{proof}

\begin{lem} \label{lc.3}
The rates of convergence for $\hat \bB$ are as follows. Two regimes are considered.

If $\lambda_m > C_1 p$ for constant $C_1 > 0$, we have \\
(i) $\|\bH^{-1}\| = O_P(1)$, \\
(ii) $\| \hat\bB - \bB \bH^{-1}\|_{\max} = O_P(\sqrt{\log p/T})$.

If $C_2\sqrt{p} (\log T)^{1/r_2} \le \lambda_m \le C_1 p$ for constant $C_2 > 0$, we have \\
(i') $\|\bH'\bH - \bI_m\| = O_P(c_m + 1/\sqrt{\lambda_m} + 1/\sqrt{T})$, \\
(ii') $\| \hat\bB - \bB \bH'\|_{\max} = O_P(\sqrt{\log p/T})$.
\end{lem}
\begin{proof}
(i') From Lemma \ref{lc.2} (i) we have
$$
\|\bF' (\hat\bF - \bF\bH')\|_F \le O_P\Big(\frac{1}{\lambda_m}\Big) \Big(\frac{1}{T} \|\bF'\bU' \bU\bF\| + \frac{1}{T}\|\bF'\bU' \bU\|\|\hat\bF -\bF\bH\| +2 \|\bF \bB' \bU \| \Big)\,.
$$
We claim $\|\bF'\bU'\| = O_P(\sqrt{T p})$. Hence,$\|\bF' (\hat\bF - \bF\bH')\|_F = O_P(p/\lambda_m) + O_P(T/\sqrt{\lambda_m})$. With this, we bound $\|\bH'\bH - \bI_m\|$. First obviously $\|\bH\| = O_P(1)$ since $\lambda_1/\lambda_m$ is bounded. Then from \cite{FanLiaMin13}, we know
\begin{align*}
\|\bH'\bH - \bI_m\| & \le \frac{1}{T} \|\bF'(\hat\bF - \bF\bH')\| (1+\|\bH\|) + \|\bH\|^2\|\bF'\bF/T - \bI_m\| \\
& = O_P(c_m + 1/\sqrt{\lambda_m} + 1/\sqrt{T})\,.
\end{align*}
It remains to show that $\|\bF'\bU'\| = O_P(\sqrt{Tp})$. By definition,
$$
\|\bF'\bU'\bU\bF\| = \|\bU\bF \bF'\bU'\| = \sup_{\bx \in S^{p-1}} \|\bF'\bU'\bx\|^2 \le 2 \sup_{\bx \in \mathcal N}  \|\bF'\bU'\bx\|^2 \,,
$$
where $\mathcal N$ is a $1/4$-net of the unit sphere $S^{p-1}$ and $|\mathcal N| \le 9^p$. Since$\|\bF'\bU'\bx\|^2 = \sum_{k=1}^m (\sum_{t \le T} f_{kt} \bu_t'\bx)^2 \le mCT \sum_{t \le T} (\bu_t'\bx)^2$, using Chernoff bound, we have
$$
\mathbb P\Big(\|\bF'\bU'\bU\bF\| \ge t\Big) \le 9^p \cdot e^{-\frac{\theta t}{2CmT}} (\mathbb E[e^{\theta(\bu_t'\bx)^2}])^T\,.
$$
$\bu_t'\bx$ is sub-Gaussian, so choosing $t \asymp Tp$, we obtain that $\|\bU\bF\| = O_P(\sqrt{T p})$.

(i) In (i'), we showed $\|\bH'\bH - \bI_m\| = O_P(c_m + 1/\sqrt{\lambda_m} + 1/\sqrt{T})$.
If in addition, we know $\lambda_m \ge C_1 p$, then $c_m = o(1)$ so that $\|\bH'\bH - \bI_m\| = o_P(1)$. So we conclude $\lambda_{\min}(\bH'\bH) > 1/2$ with probability approaching one according to Weyl's Theorem. Thus $\|\bH^{-1}\| = O_P(1)$.

(ii) Decompose $\hat\bB - \bB \bH^{-1}$ as follows:
\begin{align*}
\hat\bB - \bB \bH^{-1} & = \frac{1}{T} \bY \hat\bF - \bB\bH^{-1} \\
& = \frac{1}{T} \bB \bH^{-1}(\bH \bF' - \hat\bF') \hat\bF + \frac{1}{T} \bU (\hat\bF - \bF\bH') + \frac{1}{T} \bU\bF\bH'\,.
\end{align*}
\cite{FanLiaMin13} showed that
$$
\frac{1}{T} \|\bU\bF\|_{\max} = \max_{i \le p, k \le m} \Big| \frac{1}{T}\sum_{t=1}^T u_{it} f_{tk} \Big| = O_P\Big(\sqrt{\frac{\log p}{T}}\Big)\,.
$$
Thus the max norm of the last term is $O_P(\sqrt{\log p/T})$. The max norms of the first  and second terms are bounded respectively by
$$
\frac{m}{T} \|\bB\|_{\max} \|\bH^{-1}\| \|\bH \bF' - \hat\bF'\|_{\max} \cdot \sqrt{T} \|\hat\bF\|_F = O_P\Big( \Big(\frac{1}{\sqrt{\lambda_m}} + \frac{1}{\sqrt{p}} + \sqrt{\frac{c_m}{T}}\Big) (\log T)^{\frac{2}{r_2}}\Big)\,,
$$
and by
\begin{align*}
& O_P\Big(\frac{1}{\lambda_m T^2}\Big) \| \bU \bU'\bU\hat\bF + \bU\bF \bB' \bU\hat\bF + \bU\bU' \bB \bF' \hat\bF\|_{\max} \\
& \le O_P\Big(\frac{1}{\lambda_m T^2}\Big) \Big( \|\bU\bU'\bU\|_{\infty} \|\hat\bF\|_{\max} + m \|\bU\bF\|_{\max} \|\bB'\bU\|_{\infty} \|\hat\bF\|_{\max} + T\|\bB'\bU\|_{\infty} \|\bU\|_{\max}  \Big)
\\
& = O_P\Big(\frac{1}{\lambda_m T^2}\Big) \Big( T\sqrt{pT} (\log T)^{\frac{1}{r_2}} + \sqrt{T\log p} (T\sqrt{\lambda_1}) (\log T)^{\frac{1}{r_2}} + T^2\sqrt{\lambda_1} (\log (pT))^{\frac{1}{r_1}} \Big)  \,.
\end{align*}

Simplify and Combine the rates together, and note $\lambda_m > C_1 p$ in this case and $\sqrt{p} (\log T)^{1/r_2} = o(\lambda_m)$, we obtain,
$$
\|\hat\bB - \bB \bH^{-1}\|_{\max} = O_P\Big( \frac{1}{\sqrt{p}} \Big((\log T)^{\frac{2}{r_2}} +(\log (pT))^{\frac{1}{r_1}}\Big)  + \sqrt{\frac{\log p}{T}} \Big) = O_P\Big( \sqrt{\frac{\log p}{T}} \Big) \,.
$$

(ii') Now let us consider the other situation. We have a different decomposition of $\hat\bB-\bB\bH'$:
\begin{align*}
\hat\bB - \bB \bH' & = \frac{1}{T} \bY \hat\bF - \bB\bH' \\
& = \frac{1}{T} \bB \bF'(\hat\bF - \bF\bH') + \bB(\frac{1}{T}\bF'\bF - \bI_m) \bH' + \frac{1}{T} \bU (\hat\bF - \bF\bH') + \frac{1}{T} \bU\bF\bH'\,.
\end{align*}
As before, the max norm of the last term is $O_P(\sqrt{\log p/T})$. The max norms of the first three terms are bounded respectively by
\[
\frac{\sqrt{m}}{T} \|\bB\|_{\max} \|\bF'(\hat\bF - \bF\bH')\|_F = O_P(\sqrt{c_m/T} + 1/\sqrt{p})\,;
\]
\[
\sqrt{m} \|\bB\|_{\max} \|\frac{1}{T}\bF'\bF - \bI_m\| \|\bH'\| = O_P(\sqrt{\lambda_1/(pT)})\,;
\]
and
\begin{align*}
& O_P\Big(\frac{1}{\lambda_m T^2}\Big) \| \bU \bU'\bU\hat\bF + \bU\bF \bB' \bU\hat\bF + \bU\bU' \bB \bF' \hat\bF\|_{\max} \\
& \le O_P\Big(\frac{1}{\lambda_m T^2}\Big) \Big( T\sqrt{pT} (\log T)^{\frac{1}{r_2}} + \sqrt{T\log p} (T\sqrt{\lambda_1}) (\log T)^{\frac{1}{r_2}} + T\|\bB'\bU\bU'\|_{\max} \Big)  \,,
\end{align*}
where $\|\bB'\bU\bU'\|_{\max} = O_P(T\sqrt{\lambda_1/p} +\sqrt{\lambda_1\log p})$ is quite small.

Simplify and Combine the rates together, we obtain,
$$
\|\hat\bB - \bB \bH'\|_{\max} = O_P\Big(\frac{\sqrt{p} (\log T)^{1/r_2}  }{\lambda_m \sqrt{T}}+ \sqrt{\frac{\lambda_1}{pT}} + \sqrt{\frac{\log p}{T}} \Big) = O_P\Big(\sqrt{\frac{\log p}{T}}\Big)\,.
$$

\end{proof}

\bf{Proof of Theorem \ref{thm_thresholding}
\begin{proof}
Recall that $\hat u_{it} = y_{it} - \hat\bb_i'\hat\bff_t$. We separately consider the two cases in Lemma \ref{lc.3}. If $\lambda_m > C_1 p$, so $\bH^{-1}$ is well defined. We have
$$
u_{it} - \hat u_{it} = (\hat\bb_i' - \bb_i'\bH^{-1})(\hat \bff_t - \bH \bff_t) + \bb_i'\bH^{-1}(\hat \bff_t - \bH \bff_t) + (\hat\bb_i' - \bb_i'\bH^{-1})\bH \bff_t\,.
$$
Therefore by Cauchy-Schwarz,
\begin{align*}
\max_{i \le p} T^{-1} \sum_{t=1}^T |\hat u_{it} - u_{it}|^2 \le & 3 \max_i \|\bb_i'\bH^{-1}\|^2 \frac{1}{T} \|\hat\bF - \bF \bH'\|_F^2 \\
& + 3 m \|\hat\bB - \bB \bH^{-1}\|_{\max}^2 \frac{1}{T} \|\hat\bF - \bF \bH'\|_F^2 \\
& + 3 m  \|\hat\bB - \bB \bH^{-1}\|_{\max}^2 \frac{1}{T} \sum_{t=1}^T \|\bH\bff_t\|^2\,.
\end{align*}
It follows from Lemma \ref{lc.2} and \ref{lc.3} (ii) that
\begin{align*}
\max_{i \le p} T^{-1} \sum_{t=1}^T |\hat u_{it} - u_{it}|^2 = O_P\Big(\frac{\lambda_1}{pT} \Big(\frac{p^2}{\lambda_m^2 T} + \frac{T}{\lambda_m}\Big) + \frac{\log p}{T} \Big) = O_P\Big(\frac{\log p}{T} + \frac1p\Big)\,.
\end{align*}
Replacing the average over $t$ in the above inequality with maximum over $t$ and $T^{-1} \|\hat\bF - \bF\bH'\|_F^2$ with $m \|\hat\bF - \bF\bH'\|_{\max}^2$, we can also derive bound for $\max_{i\le p, t\le T} |\hat u_{it} - u_{it}|$. Since from Assumption \ref{assump2:factor} we have $\max_{t \le T} \|\bff_t\| = O_P((\log T)^{1/r_2})$, we get $\max_{i\le p, t\le T} |\hat u_{it} - u_{it}| = o_P(1)$.

Now if $C_2\sqrt{p} (\log T)^{1/r_2} \le \lambda_m \le C_1 p$, we apply a different way of decomposing $u_{it} - \hat u_{it}$.
$$
u_{it} - \hat u_{it} = (\hat\bb_i' - \bb_i'\bH')(\hat \bff_t - \bH \bff_t) + \bb_i'\bH'(\hat \bff_t - \bH \bff_t) + (\hat\bb_i' - \bb_i'\bH')\bH \bff_t + \bb_i'(\bH'\bH -\bI_m) \bff_t\,.
$$
Therefore by Cauchy-Schwarz,
\begin{align*}
\max_{i \le p} T^{-1} \sum_{t=1}^T |\hat u_{it} - u_{it}|^2 \le & 4 \max_i \|\bb_i'\bH'\|^2 \frac{1}{T} \|\hat\bF - \bF \bH'\|_F^2 \\
& + 4 m \|\hat\bB - \bB \bH'\|_{\max}^2 \frac{1}{T} \|\hat\bF - \bF \bH'\|_F^2 \\
& + 4 m  \|\hat\bB - \bB \bH'\|_{\max}^2 \frac{1}{T} \sum_{t=1}^T \|\bH\bff_t\|^2 \\
& + 4 \max_i \|\bb_i\|^2 \|\bH'\bH -\bI\|  \frac{1}{T} \sum_{t=1}^T \|\bff_t\|^2 \,.
\end{align*}
It follows from Lemma \ref{lc.2} and \ref{lc.3} (ii') that
\begin{align*}
\max_{i \le p} T^{-1} \sum_{t=1}^T |\hat u_{it} - u_{it}|^2 & = O_P\Big(\frac{\lambda_1}{pT} \Big(\frac{p^2}{\lambda_m^2 T} + \frac{T}{\lambda_m}\Big) + \frac{\log p}{T} + \frac{\lambda_1}{pT} \Big) = O_P\Big(\frac{\log p}{T}+\frac1p\Big)\,.
\end{align*}
Again, it is not hard to show $\max_{i\le p, t\le T} |\hat u_{it} - u_{it}| = o_P(1)$.

Finally, Lemma \ref{lem_Basic} concludes the theorem by choosing $a_T = \sqrt{\log p/T} + \sqrt{1/p}$ in both cases.
\end{proof}

\bibliographystyle{ims}
\bibliography{Reference}

\end{document}